%% file: transfer.tex
\documentclass[12pt]{iopart}

\usepackage{graphicx}

\usepackage{amssymb}
\usepackage{amsmath}
\usepackage{algorithm}
\usepackage{algorithmic}
\usepackage{cite}
\usepackage[mathscr]{euscript}
\usepackage[T1]{fontenc}
\usepackage{matlab-prettifier}
\usepackage{subfigure}

\usepackage{filecontents}

\usepackage{xcolor}

\graphicspath{{./plots/}}
\begin{document}

\title[Randomized Misfit Approach for Inverse Problems]{A
  Data-Scalable Randomized
  Misfit Approach for Solving Large-Scale PDE-Constrained Inverse Problems}


\author{E B Le$^1$, A Myers$^1$, T Bui-Thanh$^{1,2}$ and Q P Nguyen$^3$}
\address{$^1$ Institute for Computational Engineering
  and Sciences, The University of Texas at Austin, Austin, TX, USA} 
\address{$^2$ Department of Aerospace Engineering and Engineering
  Mechanics, The University of Texas at Austin, Austin, TX, USA}
\address{$^3$ Department of Petroleum and Geosystems Engineering, The University of Texas at Austin, Austin, TX, USA}  

\ead{ \{ellenle,aaron,tanbui\}@ices.utexas.edu and quoc\_p\_nguyen@mail.utexas.edu}

\renewcommand{\thefootnote}{\arabic{footnote}}

\input tbmacrosIP.tex

\begin{abstract}
A randomized misfit approach is presented for the efficient
solution of large-scale PDE-constrained inverse problems with
high-dimensional data. The purpose of
this paper is to offer a theory-based framework for
random projections in this inverse problem setting. The stochastic approximation to
the misfit is analyzed using random projection theory. By expanding
beyond mean estimator convergence, a practical characterization of
randomized misfit convergence can be
achieved. The theoretical results developed hold with any valid random
projection in the literature. The class of feasible distributions is broad yet simple to characterize compared to previous stochastic misfit methods. This class includes very sparse random
projections which provide additional computational benefit. A different proof for a variant of the
Johnson-Lindenstrauss lemma is also provided. This leads to a different intuition
for the $\mc{O}(\varepsilon^{-2})$ factor in bounds for
Johnson-Lindenstrauss results. The main contribution of this paper is a theoretical result
showing the method guarantees a valid solution for small reduced
misfit dimensions. The interplay between Johnson-Lindenstrauss theory and Morozov's
discrepancy principle is shown to be essential to the
result. The computational cost savings for large-scale
PDE-constrained problems with high-dimensional data is discussed.
Numerical verification of the developed theory is presented for
model problems of estimating a distributed parameter in an elliptic partial
differential equation.  Results with different random projections are presented to demonstrate the viability and accuracy of the proposed approach.
\end{abstract}

\noindent{\it Keywords \/} Random projections, stochastic programming, large-scale inverse
problems, model reduction, large deviation theory,
Johnson-Lindenstrauss lemma, big data.

\ams{
35Q62,  
62F15,  
35R30,  
35Q93,  
65C60  
}


\maketitle

\section{Introduction}
\seclab{intro}
An emerging grand challenge in computational science and engineering
is the solution of large-scale statistical inverse problems
governed by PDEs that involve large amounts of observational
data. These are difficult problems that
can be found in diverse areas of science, engineering, and
medicine, ranging from inference of the basal friction field in
continental ice-sheet modeling to estimation of contaminant
plume concentration in groundwater models. They are often characterized
by infinite-dimensional parameter fields (for example, a
spatially-distributed quantity) which results in a
parameter dimension in the thousands or millions when discretized. 
We seek the Bayesian solution, as it offers a range of estimates that are consistent with data while
accounting for uncertainty in the data, model, and prior
knowledge. Unfortunately, this amounts to exploring a posterior
probability, a task that is notoriously intractable for the
problems of interest. 

The dominant cost in this setting is measured in number of PDE
solves. Each PDE solve takes minutes or hours even on modern
supercomputers (e.g. \cite{KomatitschTsuboiJiEtAl03,Bui-ThanhBursteddeGhattasEtAl12,Martin15}). State-of-the-art methods require repeated evaluations
of an objective functions and its derivative information, resulting in
a total cost of hundreds, thousands, or millions of PDE solves for
many realistic problems. Reducing the number of PDE solves
is paramount. Relative to the cost of a model run, linear algebra is considered negligible.

In this paper, we present a randomized misfit approach to directly
address the computational burden induced by high-dimensional
data. Note that the idea of randomizing a misfit function is not new. Randomized approximations of
misfit functions can be found in methods for seismic inversion
\cite{AravkinFriedlanderHerrmannEtAl12,MoghaddamHerrmannothers10,LeeuwenAravkinHerrmann11}, in
stochastic optimization algorithms such as stochastic gradient descent
(see e.g. \cite{SchraudolphGraepel03, FinkelKleemanManning08}), and in
the sample average approach (SAA) \cite{NemirovskiJuditskyLanEtAl09,ShapiroDentchevaRuszczynski09,KleywegtShapiroHomem-de-Mello02}. 

What is novel here is the particular randomized
misfit framework and the resulting analysis. It is clear that the
randomized objective function converges; it is less obvious that the
minimizer converges. A connection with random
projection theory is the key to understanding why the method
results in an acceptable solution for a surprisingly small randomized
misfit dimension, not just in the limit. This analysis potentially could be applied
to existing methods that use randomized objective functions. 




Roughly speaking, random projections are
``quasi-orthogonal'' transformations from high-dimensional spaces to
much lower-dimensional spaces that, with high probability,
preserve geometric properties such as Euclidean norms, distances, and
angles. They are particularly revered for possessing such properties {\em independent of the
  original data dimension}. The geometric invariance properties are a
consequence of the concentration of measure phenomenon in
high dimensions. One can check that two high-dimensional random normal vectors on the
unit sphere are nearly orthonormal, and that this phenomenon becomes
more pronounced as the dimension grows larger. We show that for a broad class of distributions, the probability that a sample average
falls within a specified ball around its mean grows exponentially high
with the sample size. This is the power of many independent random
projections working together. Random projections provide probabilistic
accuracy bounds that are parameterized by the degree of approximation or the dimension of
the reduced space. That is, given a tolerance of approximation, one
can find the reduced space dimension that will preserve
Euclidean norm and vice versa. To assist in the practical use and verification of the method, in our numerical
examples we use random projections that are easy to
implement. 

An active area of research is developing optimization methods for when the data set does not
even fit in memory. The data needs to be subsampled prior to input. We
stress here that this is not the main target of the randomized
misfit approach. In the approach here, the data vector is not subsampled, but rather the
misfit between the model and the data is linearly transformed to a
smaller dimension where its geometric properties are preserved. This
is equivalent to summing random linear combinations of the misfit
components. We
are not cleaning the data, fusing data points, or choosing a random subset of data to
represent the full data set. We use the entire data set. The
motivation is that the dominant cost in our problem setting is the
number of PDE solves. The misfit vector dimension, as we will show, is
a hard upper bound on a factor of the
dominant cost. Thus if we can transform the misfit to a smaller
dimension, we can reduce the dominant cost of solving the inverse problem and
guarantee the accuracy of the solution. Computational cost is
discussed in detail in Section \secref{costanalysis}.

The presentation is here is purposefully general and does not assume
any particular underlying
structure of the observational data, aside from its relationship to
parameter space via the parameter-to-observable map and the noise
model. Again there is a large body of
work in data sampling, compression and/or fusion that exploits known
underlying structure of the observational data set, typically for
specific inverse problems. These methods are not incompatible with the approach we
outline. They could potentially be combined with the method here to provide
maximum computational savings.

\subsection{Current state-of-the-art and our contributions}
\seclab{existingwork}

To keep the discussion succinct and relevant to the problem of
interest, where the dominant cost is PDE solves, the review is
limited to existing work in randomized methods for PDE-constrained
inverse problems. An active area of research is in applying random projections to linear
regression problems. The dominant cost in these problems is generally
measured in linear algebraic operations.

Since \cite{HalkoMartinssonTropp11}, many randomized methods to reduce the computational
complexity of large-scale PDE-constrained inverse problems have
focused on use of the randomized SVD algorithm of
\cite{MartinssonRokhlinTygert11a}. This algorithm has been used to generate truncated SVD approximations of the
parameter-to-observable operator \cite{AlexanderianPetraStadlerEtAl16,IsaacPetraStadlerEtAl15,XiangZou15, XiangZou13,ChaillatBiros12}, the
regularization operator \cite{LeeKitanidis14, KitanidisLee14}, or the
prior-preconditioned Hessian of the objective function
\cite{Bui-ThanhGhattasMartinEtAl13, SaibabaKitanidis15,
  AlexanderianPetraStadlerEtAl14, Bui-ThanhGirolami14,
  Bui-ThanhBursteddeGhattasEtAl12}. The algorithm uses a random projection matrix to produce a
low-rank operator. To our knowledge, only Gaussian distributions are used. The randomized operator is subsequently
factored to generate an approximate SVD
decomposition for the original operator $\mc{A}$. Theoretical results in
\cite{MartinssonRokhlinTygert11a} guarantee the spectral norm accuracy of this
approximation is of order $\sigma_{k+1}(\mc{A})$ with a very high
user-defined probability. Here $k$ is equal to the reduced dimension $n$ plus a small number of 
oversampling vectors. Subsequently, results known about
the accuracy of a deterministic inverse solution (e.g., Proposition 1 in
\cite{SaibabaLeeKitanidis15}, Theorem 1 in \cite{XiangZou13})  to a
problem approximated with a randomized method are derived using this
bound from \cite{MartinssonRokhlinTygert11a}. The bounds
assume knowledge of  $\sigma_{k+1}(\mc{A})$.

{\em
Random source encoding} or {\em
simultaneous (random) source} methods have been shown to be effective for
parameter estimation in PDE-constrained inverse problems with multiple
right-hand sides (sources) and corresponding data sets
\cite{Roosta-Khorasani15,Roosta-KhorasaniVanAscher14,Roosta-KhorasaniAscher15,Roosta-KhorasaniDoelAscher14,Roosta-KhorasaniSzekelyAscher15,RouthLeeNeelamaniEtAl14,KrebsAndersonNeelamaniEtAl12,NeelamaniKrohnKrebsEtAl10, KrebsAndersonHinkleyEtAl09,LeeuwenAravkinHerrmann11,HaberChungHerrmann12}. This
problem framework characterizes many inverse problems, including electromagnetic
imaging (e.g. \cite{HaberAscherOldenburg04,
  OldenburgHaberShekhtman12}), seismic waveform inversion
(e.g. \cite{Pratt99, VirieuxOperto09, KrebsAndersonHinkleyEtAl09,HerrmannErlanggaLin09}), the DC
resistivity problem
(e.g. \cite{HaberHeldmannAscher07,HaberHeldmannAscher07}), and
electromagnetic impedance tomography
(e.g. \cite{DuraiswamiSarkarChahine98} or Sec. 6.3 in \cite{KaipioSomersalo05}). Simultaneous
source methods take random linear combinations of $s$ sources to
produce $\tilde{s}$ randomly combined sources, where $\tilde{s} \ll s$.
 The result is a randomized misfit function that requires just $\tilde{s}$ PDE solves to evaluate instead of $s$
PDE solves. 
The work in \cite{YoungRidzal12}
 shows that source encoding in its stochastic
reformulation (and as a stochastic trace estimator method \cite{Hutchinson90}) is equivalent
to an application of the random projection defined in \cite{Achlioptas03}. 
Simultaneous source methods point out that numerical solutions are
surprisingly better
than the theory predicts with a small number of sources $\tilde{s}$
(e.g. $\tilde{s} \sim\mathcal{O}(1)$)
\cite{AravkinFriedlanderHerrmannEtAl12,HaberChungHerrmann12,YoungRidzal12,
  Roosta-Khorasani15,Roosta-KhorasaniVanAscher14,Roosta-KhorasaniAscher15,Roosta-KhorasaniDoelAscher14,Roosta-KhorasaniSzekelyAscher15,KrebsAndersonHinkleyEtAl09}.

This paper extends the above work in several directions. The approach
outlined here allows for a stochastic reformulation of all
PDE-constrained inverse problems recast in a constrained least-squares
formulation, not just multi-source problems. Our analysis of
computational efficiency is necessarily different and depends on how
the large data dimension affects the optimization. The
cost savings for this method is subtler than the cost savings in simultaneous source
methods. 

The main contribution of this paper is the first theoretical result to
guarantee that the deterministic solution obtained with the randomized cost is valid
for a fixed small reduced dimension. This provides an explanation
for the surprising quality of solutions when using a
randomized misfit function with a small reduced dimension. The
efficacy of the randomized misfit approach is a result of the interplay between
Morozov's discrepancy principle and random projection theory. As will be shown, the data error and ill-posedness inherent in
inverse problems is what allows random projections to be successful. 

The key is the use of large deviation theory to arrive at a practical
characterization of objective function convergence. Large deviation
techniques form the basis of random projection theory and other
effective randomized dimension reduction algorithms (e.g. randomized SVD). This approach
exploits the concentration of measure phenomenon in high dimensions. We do
not depend on the slow convergence of estimators to their exact
mean. Instead it is shown that for a certain class of distributions, the
tail probability of a sample average of misfit estimators decays
exponentially, with a rate parameterized by the sample size. This
statement is turned into a probabilistic bound on the randomized cost for a fixed sample size. 

The class with sufficient large deviation decay rate
turns out to be {\em subgaussian} random variables. This class
contains many of the
distributions used in the simultaneous source literature. Another
novel aspect of the approach here is that it permits the use of any random projection in the
literature. Many new random projections have appeared since the
seminal work in \cite{Achlioptas03}. Potentially useful
random projections are those drawn from very sparse distributions,
which we test in our examples. 
 
Additionally, from the stochastic formulation of the misfit, a different version of the Johnson-Lindenstrauss
embedding theorem \cite{Dirksen15,Matousek08,IndykNaor07} is
shown. This leads to an insight into why the reduced misfit dimension $n$ is
$\mc{O}(\varepsilon^{-2})$, where $\varepsilon$ is the relative error
of the randomized cost function.

The structure of the paper is as follows. Section \secref{analysis}
presents the theoretical analysis for the randomized misfit approach
by deriving the large deviation bounds on the objective function error for a broad class
of distributions. The reduced misfit dimension is shown to be
independent of the original data dimension. This derivation leads to a
different proof of a variant of the celebrated Johnson-Lindenstrauss embedding
theorem.  Using Morozov's discrepancy principle, Theorem
\theoref{morthm} shows that the effective reduced misfit dimension is
also bounded below by the noise in the problem. Therefore, the RMA
solution is a guaranteed solution for the original problem with a high
user-defined probability. The reduced computational cost in problems
with high-dimensional data is assessed in Section \secref{costanalysis}. Section \secref{numresults} summarizes numerical
experiments on a model inverse heat conduction problem in one-, two-,
and three spatial-dimensions. We compare the RMA solution obtained
with different distributions to the solution of the
full problem. We also provide numerical support for Theorem \theoref{morthm}.

\section{The randomized misfit approach for inverse problems}

We assume an additive noise-corrupted pointwise
observational model
\begin{equation}
\eqnlab{pointwiseObs}
  \d_j = \w\LRp{\x_j;\u} + \eta_j, \quad j = 1,\ldots,N,
\end{equation} 
where the objective is to reconstruct the distributed parameter $\u$
given $N$ data points $\d_j$, with $N$ large.  For a given $\u$, a set
of states $\w\LRp{\x_j;\u}$ is obtained by evaluating an
expensive-to-solve forward model governed by PDEs, and then applying
a linear observation operator to match the data locations. The location of
an observational data point in an open and bounded spatial domain
$\Omega$ is denoted by $\x_j$, and $\eta_j$ is assumed to be Gaussian random
noise with mean 0 and variance $\sigma^2$.

Concatenating the observations, we rewrite \eqnref{pointwiseObs} as
\begin{equation}
  \eqnlab{observation}
  \db = \F\LRp{\u} + \etab, 
\end{equation}
where $\F\LRp{\u} := \LRs{\w\LRp{\mb{x}_1;\u},
  \ldots,\w\LRp{\mb{x}_N;\ub}}^\top$ is the parameter-to-observable
map. Although
the forward problem is usually well-posed, the inverse
problem is ill-posed. An intuitive reason is that discrete
observations can only contain limited information about an
infinite-dimensional parameter. The more complete explanation is that the
parameter-to-observable map exhibits rapid spectral decay. This can be
numerically observed and is proven for many practical inverse problems
\cite{Bui-ThanhGhattas12a, Bui-ThanhGhattas12,
  Bui-ThanhGhattas13a}. By an application of Picard's Theorem we may
then show the inverse operator is unbounded and therefore the problem is ill-posed.   

A standard deterministic Tikhonov approach resolves the
ill-conditioning by adding a quadratic term to the cost function, so
that the problem may now be formulated as
\begin{equation} 
\min_\u\enspace\mathscr{J} \LRp{\u} := \frac{1}{2}\nor{\yobsh -
  \fmaph\LRp{\u}}^2+ \half\nor{R^\half \u}^2,
\end{equation}
where $\yobsh - \fmaph\LRp{\u}:=\frac{1}{\sigma} \LRs{
  \db-\fmap(\u) }$ is the {\em data misfit vector}, Euclidean norm in $\R^N$ is denoted by
$\nor{\cdot}$, and $\vectornorm{R^\half\cdot}$ is a norm weighted by
a regularization operator $R$. 
This point estimate does not account for the uncertainty in the
solution. Thus we recast the problem in the framework of Bayesian
inference, 
where we seek a statistical description of all possible
parameter fields that are consistent with the observations. The
Bayesian solution is a probability distribution that accounts for the uncertainties in the
observations, the forward model, and the prior knowledge.
It requires specification of a likelihood model, which
characterizes the probability that the parameter $\u$ could have
produced the observed data $\yobs$. It also requires a prior model, which is
problem-dependent and represents a subjective belief regarding the
distribution of $\u$. The prior model must ensure sufficient 
regularity of the parameter so that the problem is well-posed \cite{Stuart10,Bui-ThanhGhattasMartinEtAl13, PetraMartinStadlerEtAl14}. 

The additive-noise model \eqnref{observation} is used to
construct the likelihood pdf which is expressed as
\begin{equation}
  \like \propto \exp\LRp{\frac{1}{2}\nor{\yobsh -
  \fmaph\LRp{\u}}^2}.
\end{equation}

For concreteness of presentation, we postulate that the
prior is a Gaussian random field with mean $\u_0$  and a covariance
operator $\mc{C}$.  
We must stress here that the choice of a meaningful prior in the
infinite-dimensional setting is an active area of research
\cite{Stuart10, DashtiHarrisStuart12,LassasSaksmanSiltanen09,Bui-ThanhGhattas15}. The Gaussian prior used
here is chosen only to ensure well-posedness and be
computationally amenable to general large-scale problems. We
choose $\mc{C}=\mc{A}^{-2}$, where $\mc{A}$ is a
Laplacian-like operator with its domain of definition specified by an
elliptic PDE, appropriately-chosen boundary conditions, and parameters
than can encode spatial correlation and anisotropy information (for
specific implementation details see \cite{Bui-ThanhGhattasMartinEtAl13, AlexanderianPetraStadlerEtAl16}). This choice avoids
constructing and inverting a dense covariance matrix and exploits existing fast
solvers for elliptic operators. It additionally provides a connection
to the Mat\'{e}rn covariance functions used frequently in
geostatistics \cite{LindgrenRueLindstroem11, SaibabaKitanidis15,Kitanidis12} and therefore has a
scientific justification. Note that directly specifying a
covariance function and then factorizing the covariance matrix is
common and reasonable for small- to medium-scale statistical inverse
problems, but is intractable for large-scale problems \cite{Martin15}.

We then discretize the prior,
the forward equation, and the parameter $\u$ (yielding a finite-dimensional
vector $\ub \in \R^m$)  through the finite element method (see \cite{Bui-ThanhGhattasMartinEtAl13,
  PetraMartinStadlerEtAl14} for a comprehensive treatment) so that the
finite-dimensional posterior probability of $\ub$ is given by Bayes formula as
\begin{equation}
  \eqnlab{post}
  \post \propto \exp\LRp{\frac{1}{2}\nor{\yobsh - \fmaph\LRp{\ub}}^2 + \half \nor{\ub-\ub_0}^2_{\mc{C}}}
\end{equation}
where $\nor{\cdot}_{\mc{C}} := \nor{\mc{C}^{-\half}\cdot}_\Ltwoc$ denotes the weighted $\Ltwo$ norm induced by the $\Ltwo$ inner product $\LRa{\cdot,\cdot}_\Ltwo$.
The maximum a posteriori (MAP) point of \eqnref{post} is defined as
\begin{equation}
  \eqnlab{MAP}
 \umap := \arg\min_\ub \J\LRp{\ub, \db} =\half \nor{\yobsh -
  \fmaph\LRp{\ub}}^2 + \half \nor{\ub-\ub_0}^2_{\mc{C}}.
\end{equation}

Note that the last term in \eqnref{MAP} may be viewed as a Tikhonov
regularization term, and subsequently the MAP point may be considered
as a solution to the deterministic inverse problem with a
regularization ``inspired'' by the prior. Understanding the MAP point
in a Bayesian framework allows one to account for the subjectivity of
choosing a prior. Ultimately, the goal is to find the Bayesian solution which offers a
statistical description of {\em all solutions} consistent with the
data. For this paper, we restrict ourselves to MAP computation, a necessary starting
point, in order to focus on
methodology development in addressing the challenge of big data, i.e.,
large $N$. Scalability and efficiency of the method in the Bayesian setting is the focus of ongoing work.

The main idea of the randomized misfit approach is the following.
Let $ \rb \in R^N$ be a random vector with mean zero and identity
covariance, i.e. $\Ex_{\rb} \LRs{\rb\rb^\top} = \I$ (equivalently, let $\rb$ be
the vector of $N$ i.i.d. random variables $\zeta$ with mean zero and variance 1).

Then the misfit term of \eqnref{MAP} can be rewritten as:
\begin{equation}
\nor{\misfitonly}^2=
  \LRp{ \misfitonly }^\top 
  \Expect_{\rb}\LRs{\rb\rb^\top} \LRp{\misfitonly} =   \Expect_{\rb}\LRs{\rb^\top\LRp{\misfitonly}}^2,
\end{equation}
which allows us to write the objective functional in \eqnref{MAP} as
\begin{equation}
  \J \LRp{\ub}  = \halft\Expect_{\rb}\LRs{\rb^\top\LRp{\misfitonly}}^2 + \halft
  \nor{\ub - {\ub_0}}^2_{\Cb}. 
\end{equation}

We then approximate the expectation
$\Expect_{\rb}\LRs{\cdot}$ using a Monte Carlo approximation (also
known as the Sample Average Approximation (SAA)
\cite{NemirovskiJuditskyLanEtAl09,ShapiroDentchevaRuszczynski09}) with
$n$ i.i.d. draws $\LRc{\rb_j}_{j = 1}^n$. This leads to the randomized inverse problem
\begin{align}
  \eqnlab{randprob} 
  \min_\ub\enspace {\J}_n\LRp{\ub;\rb} &= \frac{1}{2n}\sum_{j=1}^n\LRs{\rb_j^\top\LRp{{\yobsh} - \fmaph\LRp{\ub}}}^2 + \halft \nor{\ub - {\ub_0}}^2_{\Cb}. \nonumber \\
         &= \frac{1}{2}\nor{\yobsbar - \fmapbar\LRp{\ub}}^2 + \halft \nor{\ub - {\ub_0}}^2_{\Cb},
\end{align} 
where $\yobsbar :=
\frac{1}{\sqrt{n}}\LRs{\rb_1,\ldots,\rb_n}^\top\yobsh\textrm{, and }
 \fmapbar\LRp{\ub} := \frac{1}{\sqrt{n}}  \LRs{\rb_1 ,\ldots,\rb_n
 }^\top \fmaph\LRp{\ub} \in \R^n$. We call $\yobsbar -
 \fmapbar\LRp{\ub}$ the {\em reduced data misfit vector}.

For a {\em reduced misfit vector dimension} $n \ll N$, 
we call this randomization the
randomized misfit approach (RMA).  The new problem
\eqnref{randprob} with fixed i.i.d. realizations $\LRc{\rb_j}_{j = 1}^n$ may
be solved using any scalable robust optimization algorithm. For the numerical experiments in
Section \secref{numresults}, a globalized inexact Newton-CG implementation
\cite{BranchColemanLi99} is used. The use of a similar mesh-independent Newton-type method is
assumed for the discussion of computational complexity in Section
\secref{costanalysis}.

We define the MAP point of
\eqnref{randprob} as
\begin{equation}
  \eqnlab{utmap}
  \umap_n := \arg\min_\ub \Jn\LRp{\ub},
\end{equation}
the optimal RMA cost as $\Jn^\star := \Jn\LRp{\umap_n}$, and
the optimal true cost as $\J^\star := \J\LRp{\umap}$. We wish to
characterize the errors $\snor{\Jn^\star -
  \J^\star}$ and $\nor{\umap_n - \umap}$ for a given reduced misfit dimension $n$.
This is the subject of section \secref{analysis}.

\section{An analysis of the randomized misfit approach (RMA)}
\seclab{analysis}

\subsection{Validity of the RMA solution}
\seclab{validity}

For a given $\ub$ in parameter space, it is clear that $\Jn
\LRp{\ub;\rb} $ in \eqnref{randprob} is an unbiased estimator of $\J
\LRp{\ub}$. It is also clear from the Law of Large Numbers that
$\Jn\LRp{\ub}$ converges almost surely to its mean
$\J\LRp{\ub}$. However, the efficacy of the randomized misfit approach
lies in exploiting the concentration of measure phenomenon of high
dimensions, and quantifying the convergence {\em close to the
  mean}. This requires characterizing the exponential decay of the objective
function error, which is parameterized by the reduced misfit dimension
$n$. 

We first show that errors larger than
$\delta / 2$, for a given $\delta>0$, decay with a rate at least as fast
as the tail of a centered Gaussian. That is, for some distribution in \eqnref{randprob} we have
\begin{equation}
  \eqnlab{LDbound}
  \Prob\LRs{\snor{\Jn\LRp{\ub;\rb} - \J\LRp{\ub}} > \frac{\delta}{2}} \le e^{-n I\LRp{\delta}},
\end{equation}
where
\begin{equation}
  I\LRp{\delta} \ge c\frac{\delta^2}{2\theta^2}.
\end{equation}
for some $c > 0$ and some $\theta$.

This rate is sufficient to guarantee the
solution attained from the 
the randomized misfit approach is a discrepancy principle-satisfying
solution for the original inverse problem as will be shown in Theorem
\theoref{morthm}. Inequality \eqnref{LDbound} is equivalent
to the statement that $\Prob\LRs{\snor{\Jn\LRp{\ub;\rb} - \J\LRp{\ub}} > \frac{\delta}{2}}$ satisfies
a {\em large deviation principle} with {\em large deviation rate function} $I
\LRp{\delta}$ \cite{Touchette09}.

The following proposition may be viewed as a special case of Cram\'{e}r's
Theorem, which states that a sample mean of i.i.d. 
random variables $X$ asymptotically obeys a large deviation
principle with rate $I\LRp{\delta}=\sup_k\LRc{k\delta - \ln
\Expect\LRs{e^{kX}}}$ \cite{Touchette09}. However we require the exact non-asymptotic
bounds as derived here to show convergence of the RMA for $n=\mc{O}(1).$
Recall that a real-valued random variable
$X$ is {\em $\theta$-subgaussian} if there exists some $\theta>0$ such that for all $t\in \mathbb{R}$,
$\Expect \LRs{e^{tX}} \leq e^{\theta^2t^2/2}$.

\begin{proposition}
  \propolab{JLsufficient} 
The RMA error $\snor{\Jn\LRp{\ub;\rb} -\J\LRp{\ub}}$ has a tail probability that decays exponentially in
$n$ with a nontrivial large deviation rate. Furthermore, if the
 RMA is constructed with $\rb$ such that $ 2{\snor{\Jn\LRp{\ub;\rb} -
    \J\LRp{\ub}}} $ is the sample mean
of i.i.d. $\theta$-subgaussian random variables, then its large deviation rate is
bounded below by $c\frac{\delta^2}{2\theta^2}$ for some $c > 0$.
\end{proposition}

\begin{proof}
Given $\rb$, define the random variable 
\begin{equation}
\eqnlab{misfitRV}
T\LRp{\rb ; \ub}:= \LRs{\rb^\top\LRp{{\yobsh}-\fmaph\LRp{\ub}}}^2
- \nor{\yobsh - \fmaph\LRp{\ub}}^2.
\end{equation}
By a standard Chernoff bound (see, e.g.\cite{Kelly91}), we have that the RMA tail error decays exponentially as 
\begin{equation}
\eqnlab{chernoff}
\Prob \LRs{ \frac{1}{n}\sum_{j=1}^n
T\LRp{\rb_j;\ub} >\delta} \leq e^{-n I(\delta)}, 
\end{equation}
where $I(\delta) = \max_t\LRc{t\delta - \ln \Expect
   \LRs{e^{tT(\rb ; \ub)}}}$ is the large deviation rate.

The second part of the proposition follows with 
$c=1$ by bounding $ \Expect \LRs{e^{tT\LRp{\rb ; \ub}}} $ 
in \eqnref{chernoff} and computing the maximum of $t\delta - \theta^2t^2/2$.
\end{proof}

A large number of distributions are subgaussian, notably
the Gaussian and Rademacher (also referred to as Bernoulli)
distributions, and in fact any bounded
random variable is subgaussian. One class of subgaussian
distributions that provides additional computational efficiency is the following.

\begin{definition}[$\ell$-percent sparse random variables \cite{Matousek08, LiHastieChurch06}]

 Let $s=\frac{1}{1-\ell}$ where $\ell \in [0,1)$ is the level of
 sparsity desired. Then 
  \begin{eqnarray}
    \eqnlab{Li}
    \zeta & = & \sqrt{s} \begin{cases}
      +1 & \text{with probability }\frac{1}{2s},\\
      0  & \text{with probability } \ell=1-\frac{1}{s},\\
      -1 & \text{with probability }\frac{1}{2s}
    \end{cases}
  \end{eqnarray}
is a {\em $\ell$-percent sparse distribution.}
\end{definition}

Note that for $\ell
 = 0$, $\zeta$ corresponds to a Rademacher distribution, and that
 $\ell = 2/3$ corresponds to the {\em Achlioptas distribution} \cite{Achlioptas03}.
By inspection we have that $\Expect\LRs{\zeta} = 0$ and
$\Expect\LRs{\zeta^2} = 1$, and thus draws from $\zeta$ can be used in the randomized
misfit approach.

Distribution \eqnref{Li} is well-suited for the randomized misfit approach: it
is easy to implement, and the computation of the randomized misfit vector amounts to
only summations and subtractions, adding a further speedup to the method. Increasing
from $s=1$ to $s>1$ results in a $s$-fold speedup as only $1/s$
of the data is included. Note the RMA cost can be seen as the sum of $n$
random combinations from the $N$-dimensional misfit vector.
Since each random combination has a different sparsity pattern, we
effectively do not exclude any data, yet each computation requires
only $1/s$ of the data. 

We note that for the distribution \eqnref{Li}, $1\le s<\infty$, the random variable $\zeta$ distributed by
\eqnref{Li}  has\footnote[1]{Using the inequality $\LRp{2k}! \geq
  2^kk!$ and the Taylor expansion
  around $0$, we have that for $t\in (0,1]$
\begin{equation}
\Expect\LRs{e^{t\zeta}}  = \frac{1}{s} \sum^\infty_{k=0}\frac{\LRp{st^2}^k}{\LRp{2k}!} \leq
   \frac{1}{s}\sum^{\infty}_{k=0} \frac{\LRp{st^2}}{2^k k!}
   =\frac{1}{s} e^{\frac{s}{2}t^2} = e ^ {- \ln{s} + \frac{s}{2} t^2} \leq e ^ { -t^2\ln{s} + \frac{s}{2}t^2 }. 
\end{equation}
} $\Expect\LRs{e^{t\zeta}}\leq e^\frac{b^2t^2}{2}$ with
$b=\sqrt{s-2\ln s},~ \forall t\in (0,1]$ . So, we may use it in the
following theorem. 

\begin{theorem}
Define $ \vb := {\dbh} -\Fh\LRp{\ub} \in \R^N $. If $\rb$ in  \eqnref{misfitRV} has components that are $b$-subgaussian for some $b\geq1/\sqrt{2}$, then the RMA error has a large deviation rate
  bounded below by $c\frac{\delta^2}{2\theta^2}$ for $\theta = \nor{\vb}^2 / \sqrt{2}$ and some $0<c<\frac{1}{8b^4}$.
\end{theorem}
\begin{proof}
Let $\rb \in \R^N$ such that $\rb$ has i.i.d. $b$-subgaussian
components $r_i$, with $b\geq 1/\sqrt{2}$, $\Expect \LRs{r_i}
= 0$, and $\Expect\LRs{r_i^2} = 1$. 
Define $\wbd = \frac{\vb}{\nor{\vb}}$ and $X = \rb ^\top \wbd$. 
Then
\begin{equation}
\eqnlab{mgfmisfit}
\Expect \LRs{e^{tT}} = e^{- t \nor{\vb}^2} \Expect \LRs{e ^{t \nor{\vb}^2X^2}} \qquad \forall t\in \R.
\end{equation}
From \cite[Lemma 2.2]{Matousek08}, $\Expect \LRs{X^2} = 1$ and $X$ is
also $b$-subgaussian.
Then, by \cite[Remark 5.1]{IndykNaor07}, for $0 \leq t \leq \frac{1}{4b^2}$,
\begin{equation}
\eqnlab{sqrt}
\Expect \LRs{e^{tX^2}} \leq \sqrt{2}.
\end{equation}
For $0 < t \leq
  \frac{1}{4b^2\nor{\vb}^2}$, we have
\begin{align*}
\Expect \LRs{e^ {t \nor{\vb}^2 X^2}} &\leq 1 + t\nor{\vb}^2 +
t\nor{\vb}^4 \frac{\Expect \LRs{X^4}}{2} +
\sum^\infty_{k=3} \frac{\LRp{\frac{1}{4b^2}}^k \LRp{4b^2 t
                                       \nor{\vb}^2}^k \Expect \LRs{X^{2k}}}{k!} \\
&\leq 1 + t\nor{\vb}^2 + t^2 \nor{\vb}^4 \frac{\Expect \LRs{X^4}}{2} + 
\LRp{4b^2 t \nor{\vb}^2}^3\sum^\infty_{k=3}
  \frac{\LRp{\frac{1}{4b^2}}^k \Expect \LRs{X^{2k}}}{k!}\\
&\leq 1 + t \nor{\vb}^2 + t^2 \nor{\vb}^4  \frac{\Expect
  \LRs{X^4}}{2}  +\LRp{4b^2 t \nor{\vb}^2}^3 \Expect
  \LRs{e^{\frac{1}{4b^2}X^2}} \\
&\leq 1+ t \nor{\vb}^2 + t^2 \nor{\vb}^4  \frac{\Expect
  \LRs{X^4}}{2}  + 64 \sqrt{2} b^6 t^3 \nor{\vb}^6\\
&\leq 1 + t \nor{\vb}^2 + 8b^4t^2 \nor{\vb}^4 +  64 \sqrt{2} b^6 t^3
  \nor{\vb}^6\\
&\leq e^{t \nor{\vb}^2 + 8b^4t^2 \nor{\vb}^4 +  64 \sqrt{2} b^6 t^3
  \nor{\vb}^6},
\end{align*}
using \eqnref{sqrt} in the fourth inequality and \cite[p.93]{Stroock11}
in the fifth inequality.
Let $t_\star = \frac{\delta}{8b^4 \nor{\vb}^4q}$ where $q>1$. Assuming
$\nor{\vb}^2\gg \delta$, we have that 
\begin{align*}
\Expect \LRs{ e^{t^\star T}} \leq e^{8b^4
                               t_\star^2 \nor{\vb}^4 + 64\sqrt{2} b^6
                               t_\star^3 \nor{\vb}^6} = e^{\frac{\delta^2}{8b^4\nor{\vb}^4q^2} +
  \sqrt{2}\frac{\delta^3}{8b^6\nor{\vb}^6q^3}}.
\end{align*}
Then
\begin{align*}
I \LRp{\delta} &\geq \delta t_\star - \ln{\Expect \LRs{e^{t_\star T}}}
                 \geq \LRp{1-\frac{1}{q}}\frac{\delta^2}{8b^4\nor{\vb}^4q} -
  \sqrt{2}\frac{\delta^3}{8b^6\nor{\vb}^6q^3}  \geq c \frac{\delta}{\nor{\vb}^4},
\end{align*}
where $0<c<\frac{1}{8b^4}$. Taking
$2\theta^2 = \nor{\vb}^4$ concludes the proof.
\end{proof}

A sharper result can be obtained for RMA constructed with
$b$-subgaussian random variables where $b\leq 1$. Note that this
includes the distribution \eqnref{Li} with $s = 1$ (Rademacher) and $s
= 3$ (Achlioptas) by the above theorem. 
Following \cite[(5)]{IndykNaor07}, let $g$ be a standard Gaussian
random variable, independent of all other random variables. Then, we
have that for $0 < t <
\frac{1}{2\nor{\vb}^2}$,
\begin{equation}
\Expect \LRs{e^{t\nor{\vb}^2X}} \leq \Expect_g \LRs{ \prod^N_i e^{b^2 t\nor{\vb}^2 w_i^2 g^2}}
  \leq \Expect_g \LRs{e^{t \nor{\vb}^2 g^2}} = \frac{1}{\sqrt{1-2 t \nor{\vb}^2}}. 
\end{equation}
So from \eqnref{mgfmisfit} we have that 
\begin{equation}
  \Expect\LRs{e^{tT\LRp{\ub,\rb}}} \le \frac{e^{-t\nor{\vb}^2}}{\sqrt{1 - 2 t \nor{\vb}^2}} =
  e^{-t\nor{\vb}^2 -\half\ln\LRp{1 - 2 t \nor{\vb}^2}}.
\end{equation}
Then
\begin{equation}
t \delta - \ln\LRp{ \Expect \LRs{T\LRp{\ub,\rb} }} \ge t \delta +t\nor{\vb}^2 + \half\ln\LRp{1 - 2 t
  \nor{\vb}^2} =: f\LRp{t}.
\end{equation}
Computing the derivative, we have that $f\LRp{t}$ attains a maximum at
\begin{equation}
t_{\mathrm{max}} = \frac{\delta}{2\LRp{\nor{\vb}^4 + \delta \nor{\vb}^2}}.
\end{equation}
Thus, we have
 \begin{align*}
       \max f\LRp{t} &= \frac{\delta^2}{2\LRp{\nor{\vb}^4 +
           \delta \nor{\vb}^2}} + \frac{\delta}{2\LRp{\nor{\vb}^2 + \delta}} + \half
       \ln\LRp{1 - \frac{\delta}{\nor{\vb}^2 + \delta}} \\
       &= \frac{\delta ^2}{2\LRp{\nor{\vb}^4 +
           \delta \nor{\vb}^2}} - \frac{1}{4}\frac{\delta^2}{\LRp{\nor{\vb}^2 +
           \delta}^2} - \frac{1}{6}\frac{\delta^3}{\LRp{\nor{\vb}^2 +
           \delta}^3} - \cdots \\
       & = \frac{\delta^2}{4\LRp{\nor{\vb}^4 +
           \delta \nor{\vb}^2}} + \frac{1}{4}\LRc{ \frac{\delta^2}{\LRp{\nor{\vb}^4 +
             \delta \nor{\vb}^2}}- \frac{\delta^2}{\LRp{\nor{\vb}^2 +
             \delta}^2}} \\
       &- \frac{1}{6}\frac{\delta^3}{\LRp{\nor{\vb}^2 +
           \delta}^3} - \cdots 
       \ge c \frac{\delta^2}{\nor{\vb}^4},
     \end{align*}
where we employed the Taylor expansion in the second equality, and in
the last inequality $c$ is some constant less than $1/4$. Note that the last
inequality holds for $\delta \ll \nor{\vb}^2$ and taking $2 \theta^2 = \nor{\vb}^4$ concludes the proof.

The next theorem is our main result. It guarantees
with high probability that the RMA solution will be a solution of the
original problem under Morozov's discrepancy principle, for relatively
small $n$. We first need the following lemma.
\begin{lemma} 
\label{epsisom}
Let $\vb := \yobsh - \fmaph\LRp{\ub}$.  Suppose that $\rb$ is distributed such
  that the large deviation rate of the RMA error is bounded below by
  $c\frac{\delta^2}{2\theta^2}$ for some $c > 0$ and
  $\theta = \nor{\vb}^2/\sqrt{2}$. Given a cost {\em distortion tolerance} $\varepsilon>0$ and a {\em failure rate} $\beta>0$, let
  \begin{equation}
    \eqnlab{Nlb}
    n \ge \frac{\beta}{c \varepsilon^2}.
  \end{equation}
Then with probability at least $1 - e^{-\beta}$,
\begin{equation}
  \eqnlab{MISFITpreJL}
  \noindent \LRp{1-\varepsilon} \nor{\vb}^2 \le
  \frac{1}{n}\sum_{j=1}^n\LRp{\rb_j^\top\vb}^2
  \le \LRp{1+\varepsilon} \nor{\vb}^2,
\end{equation}
and hence,
\begin{equation}
    \eqnlab{preJL}
 \LRp{1-\varepsilon} \J\LRp{\ub} \le \Jn\LRp{\ub;\rb} \le \LRp{1+\varepsilon} \J\LRp{\ub}.
\end{equation}

\end{lemma}

\begin{proof}
The proof follows from setting $\delta = \varepsilon \nor{\vb}^2$ in \eqnref{LDbound}.
\end{proof}

This lemma demonstrates a remarkable fact that with $n$ i.i.d. draws
one can reduce the data misfit dimension from $N$ to $n$ while bearing a relative error of 
$\varepsilon = \mc{O} \LRp{1/\sqrt{n}}$ in the cost function, where the reduced dimension $n$ is independent of the dimension
$N$ of the data.
This idea is the basis for data-reduction techniques via variants of the
Johnson-Lindenstrauss Lemma in existing work with random
projections (see 
e.g. \cite{HolubFridrich13,LiuFieguthClausiEtAl12,FowlerDu12}). With
the connection through the randomized misfit
approach, the ubiquitous $N$-independent Monte Carlo factor $\varepsilon = \mc{O} \LRp{1/\sqrt{n}}$ in
Johnson-Lindenstrauss literature can thus be understood by reframing the
application of a random projection as a Monte Carlo method in the
form of \eqnref{MISFITpreJL}.

Unlike other applications of the Monte Carlo method, e.g. Markov chain
Monte Carlo, in which $n$ must be large to be successful, $n$ can be moderate
or small for inverse problems, depending on the noise $\etab$ in \eqnref{observation}.
In the following theorem we show this is possible via Morozov's discrepancy principle \cite{Morozov66}. To
avoid over-fitting the noise, from \eqnref{pointwiseObs} one seeks a MAP point $\umap$ such that 
$\snor{\d_j - \w\LRp{\x_j;\umap}} \approx \sigma$, i.e. $\nor{\yobsh
  - \fmaph\LRp{\umap}}^2 \approx N$. 
We say that an inverse solution $\umap$ satisfies Morozov's discrepancy principle with parameter $\tau$ if
\begin{equation}
  \nor{\yobsh - \fmaph\LRp{\umap}}^2 = \tau N
\end{equation}
for some $\tau \approx 1 $.
\begin{theorem}[Statistical Morozov's discrepancy principle]
  \theolab{morthm}
  Suppose that the conditions of Lemma \ref{epsisom} are met. If
  $\umapn$ is a discrepancy principle-satisfying solution for the RMA cost, i.e.,
  \begin{equation}
  \mc{J}_n\LRp{\umapn,\rb} := \frac{1}{n}\sum_{j=1}^n\LRs{\rb_j^\top\LRp{{\yobsh} -
             \fmaph\LRp{\umapn}}}^2= \tau 'N
  \end{equation} 
for some $\tau ' \approx 1$, then with probability at least $1 -
e^{-\beta}$, $\umapn$ is also a solution for the original problem that satisfies Morozov's discrepancy principle with parameter
  $\tau $, i.e.
  \begin{equation}
 \mc{J}\LRp{\umapn} := \nor{\yobsh - \fmaph\LRp{\umap_n}}^2 = \tau N.
  \end{equation}
  for $\tau \in \LRs{ \frac{\tau '}{1 +\varepsilon},
    \frac{{\tau '}}{1 -\varepsilon}}$.
\end{theorem}
\begin{proof}
  The claim is a direct consequence of \eqnref{MISFITpreJL}.
\end{proof}

\subsection{Other theoretical results}
\seclab{other}

We are now in the position to show a different proof of the
Johnson-Lindenstrauss embedding theorem using a stochastic programming derivation of the RMA. Following \cite{Sarlos06}, we define a map $\mc{S}$ from $\R^n$ to
$\R^N$, where $n \ll N$, to be a  Johnson-Lindenstrauss transform (JLT) if
\begin{equation}
  \LRp{1-\varepsilon} \nor{\vb}^2 \le \nor{\mc{S}\vb}^2 \le
    \LRp{1+\varepsilon} \nor{\vb}^2,
\end{equation}
 holds with some probability $p = p\LRp{n,\varepsilon}$, where $\varepsilon>0$.

\begin{theorem}[Johnson-Lindenstrauss embedding theorem
  \cite{Dirksen15,Matousek08,IndykNaor07}] \theolab{JLlemma} Suppose that $\rb$ is distributed such
  that the large deviation rate of the RMA error is bounded below by
  $c\frac{\delta^2}{2\theta^2}$ for some $c > 0$ and some $\theta$. Let $0<\varepsilon<1$, $\vb_i
  \in \R^N, i = 1,\ldots,m$, and $n = \mc{O}\LRp{\varepsilon^{-2}\ln
    m}$. Then
  there exists a map $\mc{F}:\R^N \to \R^n$ such that
  \begin{equation}
    \eqnlab{JLin}
    \LRp{1-\varepsilon} \nor{\vb_i - \vb_j}^2 \le \nor{\mc{F}\LRp{\vb_i}-\mc{F}\LRp{\vb_j}}^2 \le
    \LRp{1+\varepsilon} \nor{\vb_i - \vb_j}^2 \quad \forall i,j.
  \end{equation}
\end{theorem} 
\begin{proof}
The conditions of Lemma \ref{epsisom} hold, thus for a given $\vb
\in \R^N$, note that \eqnref{MISFITpreJL} is equivalent to
\begin{equation}
\eqnlab{JLineq}
 \LRp{1-\varepsilon} \nor{\vb}^2 \le \nor{\mb{\Sigma}\vb}^2 \le \LRp{1+\varepsilon} \nor{\vb}^2,
\end{equation}
where
\begin{equation}
  \Sigb := \frac{1}{\sqrt{n}}\LRs{\rb_1,\ldots,\rb_n}^\top.
\end{equation}
Define  $\mc{F}\LRp{\vb} := \Sigb\vb$. Inequality \eqnref{JLin} is then a direct consequence of \eqnref{JLineq} for a pair $\LRp{\vb_i,\vb_j}$ with probability at least $1 - e^{-\frac{c}{2}   n\varepsilon^2}$. 
Using an union bound over all pairs, claim \eqnref{JLin} holds for
any pair with probability at least $1 - m^{-\alpha}$ if  $ n \ge c
 \frac{\LRp{2+\alpha}}{\varepsilon^2}\ln m$.
\end{proof}

As discussed above, $\Jn\LRp{\ub;\rb}$ is an unbiased estimator of $\J\LRp{\ub}$. 
It is therefore
reasonable to expect that $\Jn^\star := \min_\ub\Jn\LRp{\ub;\rb}$ converges to $\J^\star :=\min_\ub\J\LRp{\ub}$. 
The following result \cite[Propositions 5.2 and
  5.6]{ShapiroDentchevaRuszczynski09} states that under mild conditions
$\Jn^\star$ in fact converges to $\J^\star$. It is not unbiased, but is
however downward biased.

\begin{proposition}
  Assume that $\Jn\LRp{\ub;\rb}$ converges to $\J\LRp{\ub}$ with probability $1$ uniformly in $\ub$, then $\Jn^\star$ converges to $\J^\star$ with probability $1$. 
  Furthermore, it holds that 
\begin{equation}
\Expect\LRs{\Jn^\star} \le \Expect\LRs{\J_{n+1}^\star} \le \J^\star,
\end{equation}
that is, $\Jn^\star$ is a {\em downward-biased estimator} of $\J^\star$.
  \propolab{optValConvergence}
\end{proposition}

Stochastic programming theory gives a stronger characterization of
this convergence. One can show that $\umapn$ converges weakly to 
$\umap$ with an $n^{-\half}$ rate. 
If $\J\LRp{\ub}$ is convex with finite value, then $\umapn = \umap$ 
with probability exponentially converging to $1$.
See Chapter 5 in \cite{ShapiroDentchevaRuszczynski09} for details.
For a linear forward map $\F\LRp{\ub} = \F\ub$, that is, $\J\LRp{\ub}$ is
quadratic, we can derive a bound on the solution error using the
spectral norm of $\F$.
\begin{theorem}
 \theolab{error}
 Suppose the conditions of Lemma \ref{epsisom} hold. Let $m:=\mathrm{rank}(\Fh)$. 
  Then 
\begin{itemize}
\item[i)] $ 
      \LRp{1-\varepsilon}\J^\star \leq \Jn^\star \le \LRp{1 +
        \varepsilon} \J^\star$, and 
\item[ii)] if $\F$ is linear, then with probability at least $1 -
  m^{-\alpha}$
\end{itemize}
\begin{equation}
\eqnlab{uIn}
\nor{\umapn - \umap} \le \frac{\varepsilon}{\sigma_{\mathrm{min}}^2\LRp{\Gb}}\LRp{\nor{\Fh}\nor{\umap} + \nor{\dbh}}\nor{\Fh},
\end{equation}
  where $\Gb := \LRp{\Fh^\top\Sigb\Sigb^\top\Fh + \Cb^{-1}}^{\half}$,  and $n
  = \mc{O}\LRp{\varepsilon^{-2}\LRp{2+\alpha}\ln m}$.
\end{theorem}
\begin{proof}
  The first assertion follows from \eqnref{preJL} and the definition of
  $\umapn$ \eqnref{utmap}, indeed
  \begin{equation}
  \Jn^\star = \Jn\LRp{\umapn} \le \J\LRp{\umap} \le
  \LRp{1+\varepsilon} \J\LRp{\umap}= \LRp{1+\varepsilon} \J^\star,
  \end{equation}
  and the other direction is similar.
  For the second assertion, note that $\umap$ and $\umapn$ are solutions of the following first optimality conditions
  \begin{subequations}
    \eqnlab{optimality}
    \begin{align}
      \LRp{\Fh^\top\Fh + \Cb^{-1}}\ub^\star &=  \Fh^\top\dbh + \Cb^{-1}\ub_0, \\
      \LRp{\Fh^\top\Sigb\Sigb^\top\Fh + \Cb^{-1}}\umapn &=  \Fh^\top\Sigb\Sigb^\top\dbh + \Cb^{-1}\ub_0.
    \end{align}
  \end{subequations}

  Define $\bs{\Delta} := {\umap-\umapn}$.  
  An algebraic manipulation of \eqnref{optimality} gives
  \begin{equation}
  \LRp{\Fh^\top\Sigb\Sigb^\top\Fh + \Cb^{-1}}\bs{\Delta} =  \LRp{\Fh^\top\Sigb\Sigb^\top\Fh - \Fh^\top\Fh}\umap + 
  \Fh^\top\dbh - \Fh^\top\Sigb\Sigb^\top\dbh.
  \end{equation}
  Taking the inner product of both sides with $\bs{\Delta}$ we have
 \begin{multline}
    \eqnlab{energy}
    \LRa{\bs{\Delta}, \LRp{\Fh^T\Sigb\Sigb^T\Fh + \Cb^{-1}}\bs{\Delta}} = \LRa{\Fh\bs{\Delta}, \Sigb\Sigb^T\Fh\ub^\star - \Fh\ub^\star} \\
    + \LRa{\Fh\bs{\Delta}, \dbh - \Sigb\Sigb^T\dbh}.
  \end{multline}
   Then we can bound the left-hand side of \eqnref{energy}:
  \begin{equation}
    \eqnlab{energyLHS}
    \LRa{\bs{\Delta}, \LRp{\Fh^\top\Sigb\Sigb^\top\Fh + \Cb^{-1}}\bs{\Delta}} \ge \sigma_{\mathrm{min}}^2\LRp{\Gb} \bs{\Delta}^2. 
  \end{equation}
  To bound terms on right hand side of \eqnref{energy}, we need the following straightforward variant of \eqnref{JLineq}, i.e. $\forall \vb \in \R^N$ and $n = \mc{O}\LRp{\varepsilon^{-2}}$:
  \begin{equation}
    \eqnlab{postJL}
    \nor{\Sigb\Sigb^\top\vb - \vb} \le \varepsilon\nor{\vb}.
  \end{equation}
  Using the Cauchy-Schwarz inequality we have
  \begin{subequations}
    \eqnlab{energyRHS2}
    \begin{align}
      \LRa{\Fh\bs{\Delta}, \Sigb\Sigb^T\Fh\ub^\star - \Fh\ub^\star} &\le \varepsilon\nor{\Fh}^2\nor{\bs{\Delta}}\nor{\ub^\star},\\
      \LRa{\Fh\bs{\Delta}, \dbh - \Sigb\Sigb^T\dbh} &\le \varepsilon \nor{\Fh}\nor{\bs{\Delta}}\nor{\dbh},
    \end{align}
      \end{subequations}
      where we have used \eqnref{postJL} and definition of matrix norm.
      Next, combining \eqnref{energyRHS2} and \eqnref{energyLHS} ends the proof.
  
     Note that for inequalities in \eqnref{energyRHS2} to be valid,
      it is sufficient to choose $n, \alpha, \varepsilon$  such that
      \eqnref{postJL} is valid for $m$ basis vectors spanning the
      column space of $\Fh$, and hence $n =
      \mc{O}\LRp{\varepsilon^{-2}\LRp{2+\alpha}\ln m}$ by the union
      bound. 
\end{proof}

\begin{remark} The bound in \eqnref{uIn} is not a unique estimation. One can first rewrite $\J\LRp{\ub}$ and $\Jn\LRp{\ub;\rb}$ as
  
  \begin{eqnarray*}
    \J\LRp{\ub} &= \halft\nor{ 
      \LRs{
        \begin{array}{c}
          \dbh\\
          \Cb^{-1/2}\ub_0
        \end{array}
      }
      -\LRs{\begin{array}{c}
          \Fh\\
          \Cb^{-1/2}
        \end{array}
      }
      \ub }^2, \\
    \Jn\LRp{\ub;\rb} &= \halft\nor{ 
      \LRs{\begin{array}{cc}
          \bs{\Sigma}^\top & 0\\
          0 & \I
      \end{array}}
      \LRc{
        \LRs{
          \begin{array}{c}
            \dbh\\
            \Cb^{-1/2}\ub_0
          \end{array}
        }
        -\LRs{\begin{array}{c}
            \Fh\\
            \Cb^{-1/2}
          \end{array}
        }
        \ub }}^2.
  \end{eqnarray*}

  If  $\Sigb$ is a Johnson-Lindenstrauss transform, then 
  $\mc{S}:=\left[\begin{array}{cc} \bs{\Sigma}^\top & 0\\ 0 &
      \I \end{array}\right]$ is also a JLT with the same parameters: 

  \begin{equation}
  \nor{\mc{S} \left[\begin{array}{c}
        \vb\\
        \wbd \end{array}\right] }^2 = \nor{\bs{\Sigma} \vb}^2 + \nor{\wbd}^2 \leq
  (1+\varepsilon)\nor{\vb}^2 +\nor{\wbd}^2 \leq
  (1+\varepsilon)\nor{\left[\begin{array}{c}
        \vb\\
        \wbd \end{array}\right] }^2.
  \end{equation}
  Applying \cite[Theorem 12]{Sarlos06}, we conclude that with probability at least $1/3$,
  \begin{equation}
  \nor{\umapn - \umap} \le \frac{\varepsilon}{\lambda_{min}}\sqrt{\J^\star},
  \end{equation}
  where $\lambda_{min}$ is the minimum nonzero singular value of $\LRs{\begin{array}{c}
      \Fh^\top,
      \Cb^{-1/2}
    \end{array}
  }^\top$.
\end{remark}

\subsection{Data-scalability and cost complexity estimate}
\seclab{costanalysis}


This section presents a qualitative discussion of the computational complexity and
scalability of the randomized misfit approach. Numerical evidence
of scalability to large data dimensions is presented in Section \secref{scalability}.
For concreteness and ease of comparison, 
a Newton-type optimization method is assumed. The theory in Sections
\secref{validity} and \secref{other} is independent of the solver used. 

The cost complexity of solving the randomized problem \eqnref{randprob} is
measured in number of PDE solves, i.e. solves of the forward or adjoint PDE and
incremental variants. This characterization of complexity is agnostic to the
specific governing forward PDE or PDE solver. For nontrivial forward
problems, the total runtime of
MAP point computation and uncertainty quantification is overwhelmingly
dominated by the PDE solves; the cost of linear algebra is negligible in comparison \cite{Bui-ThanhBursteddeGhattasEtAl12,Bui-ThanhGhattasMartinEtAl13,PetraMartinStadlerEtAl14,IsaacPetraStadlerEtAl15,AlexanderianPetraStadlerEtAl16}. 

In particular, with an inexact Newton-CG method, the cost of
each Newton step is dominated by conjugate gradient (CG)
iterations. Each CG iteration requires an application of the data
misfit Hessian, which in turn requires a pair of incremental
forward and adjoint PDE solves
\cite{Bui-ThanhBursteddeGhattasEtAl12,FlathWilcoxAkcelikEtAl11,Bui-ThanhGhattasMartinEtAl13,IsaacPetraStadlerEtAl15,AlexanderianPetraStadlerEtAl16}. Thus
the total work estimate is
$\mc{O}(2rk_\mathrm{Newton})$ PDE solves. Here, $k_\mathrm{Newton}$ is the total number of Newton
iterations and $r$ is
the numerical rank of the prior-preconditioned data misfit
Hessian (or equivalently, the dimension of the likelihood-informed
subspace (LIS) of parameter space \cite{CuiMartinMarzoukEtAl14a}). Current state-of-the-art implementations demonstrate that, for a wide
class of inverse problems, the number of outer Newton iterations
$k_\mathrm{Newton}$ and the numerical
rank $r$ are both independent of the mesh-size
\cite{FlathWilcoxAkcelikEtAl11,Bui-ThanhGhattasMartinEtAl13,IsaacPetraStadlerEtAl15}. Mesh-independence
is essential for ensuring scalability of a method to very high
parameter dimensions.  

The challenge is that even though $r$ may be independent of the mesh,
it still depends on the information content of the data. For many
practical large-scale problems with high-dimensional data, $r$ is on the
order of hundreds or thousands (e.g. $r=5000$ in \cite{IsaacPetraStadlerEtAl15a} and $r=1500$ for a linear 3D
convection diffusion problem in
\cite{FlathWilcoxAkcelikEtAl11}). Consequently, even with the best methods and modern supercomputers, solving the inverse problem is
still computationally expensive.

Recall that for a given inverse problem, $r$ is a fixed constant
intrinsic to the misfit function, as it is the numerical rank of the
prior-preconditioned Hessian of the misfit. A
Newton-type method requires $2r$ PDE solves (i.e. $r$ inner
iterations) at each outer iteration to sufficiently capture the $r$
dominant modes of the misfit Hessian. Arbitrarily taking a much
smaller number of inner iterations than $r$ would result in more
Newton iterations and degradation of the overall
convergence. This constraint necessitates the use of a {\em
  surrogate misfit function}, with a Hessian that has numerical rank
smaller than $r$, in order to bypass the impact of $r$ on the overall
cost of solving the inverse problem.

Ideally, this surrogate would leverage a small loss in the ``level
of parameter information in data'' to obtain a large
reduction in the overall computational cost of computing the inverse solution. In
fact, this is what the randomized misfit approach can offer. The RMA
cost is a surrogate cost that reduces the factor of $r$
in the work estimate to an $n \ll r$, while providing a guarantee of
solution viability. Note that the reduced misfit vector
dimension $n$ is a hard upper bound on the numerical rank of the misfit Hessian
for the RMA cost $\Jn$ \eqnref{randprob}. This is numerically
demonstrated for an elliptic inverse problem in Section \secref{scalability}.
 
Using the theory in Section \secref{validity}, we can explicitly quantify
the substantial gains in computational efficiency that are achieved
with a specified accuracy level and a specified {\em confidence}
level. This occurs by reframing the deterministic solution as one that
holds with a given high probability. 

The overall work estimate for the randomized misfit approach therefore is
$\mc{O}(2nk_\mathrm{Newton})$. This cost reduction analysis
is markedly different from the analysis in the stochastic simultaneous source methods described in Section \secref{existingwork}. By combining a large
number of input sources $s$
into a smaller number $\tilde{s}$, stochastic methods for multiple sources reduce the original problem from
$\mc{O}(2rk_\mathrm{Newton}s)$ to
$\mc{O}(2rk_\mathrm{Newton}\tilde{s})$ where $\tilde{s} \ll s$. Note
that the RMA can provide a reduced work estimate in the most general
class of inverse problems where $s=1$
{\em and} a guarantee of solution viability, whereas
randomized simultaneous source methods cannot.

\section{Numerical experiments}
\seclab{numresults}

In this section we demonstrate the randomized misfit approach with different distributions for $\rb$ in
\eqnref{randprob}. We also verify that the convergence is indeed
$\mc{O}(1/\sqrt{n})$ as guaranteed by Theorem \theoref{JLlemma}. Lastly we
verify Theorem \theoref{morthm}, the statistical Morozov's discrepancy principle.

The distributions that we test with the randomized misfit approach are:
\begin{itemize}
\item Gaussian
\item Rademacher
\item Achlioptas
\item $95\%$-sparse ($s$-sparse \eqnref{Li} with $s=20$)
\item $99\%$-sparse ($s$-sparse \eqnref{Li} with $s=100$)
\item Uniform $\mc{U}\LRs{-\sqrt{12}/2,\sqrt{12}/2}$
\end{itemize}

There are many other
distributions suitable for RMA in the literature on
Johnson-Lindenstrauss transforms that we do not consider, particularly
the Subsampled Randomized Hadamard Transform of \cite{AilonChazelle09,Tropp10}
and its subsequent fast and sparse variants. These will be tested in
future work. 

We remark that subsampling (random subset) matrices are not proper
random projection matrices and thus are not suitable for use in the RMA \eqnref{randprob}. Random source encoding methods often test subsampling matrices to
reduce the dimension of the misfit
\cite{DoelAscher12,Roosta-KhorasaniDoelAscher14,Roosta-KhorasaniDoelAscher14}. For
many inverse problems with identifiable structure in the data (e.g. 3-D
hydraulic tomography \cite{CardiffBarrashKitanidis13}), subsampling can be
extremely effective for reducing the computational burden of large observational
datasets. However, in the RMA,
subsampling down to misfit dimension $n$ is equivalent to choosing
$\rb_j$ from the canonical set $\LRc{\eb_1,\ldots,\eb_N}$ {\em
  without} replacement. Therefore the set $\LRc{\rb_j}_{j =
  1}^n$ is not an i.i.d set. Similar to
\cite{DoelAscher12,Roosta-KhorasaniVanAscher14,Roosta-KhorasaniDoelAscher14},
our numerical results (omitted here) are poorer with subsampling matrices compared to results
with proper random projections. This is consistent with the idea discussed in Section \secref{intro} that
random projections are {\em geometry preserving transformations}, and
can preserve the geometric relationship between any large observational
data set and the parameter-to-observable map. Random
subset matrices do not possess this property in general.

For our model problem we consider the estimation of a distributed coefficient in an
elliptic partial differential equation. This Poisson-type problem arises in
various inverse applications, such as the heat conductivity or
groundwater problem, or in finding a membrane with a given
spatially-varying stiffness. 

For concreteness we consider the heat conduction problem on an open
bounded domain $\Omega$, governed
by  
\begin{eqnarray}
-\Div\LRp{e^\ub\Grad \w} &= 0 & \mathrm{ in } \, \Omega \nonumber\\
-e^\ub\Grad \w \cdot \mb{n} &= Bi \, \w \qquad &\mathrm{ on } \, \partial \Omega \setminus \Gamma_{R}, \eqnlab{pde}\\
 -e^\ub\Grad \w \cdot \mb{n} &= -1 &\mathrm{ on } \, \Gamma_{R}, \nonumber
 \end{eqnarray}
where $\ub$ is the logarithm of distributed thermal conductivity, $\w$
is the distributed forward state (temperature), $\mb{n}$ is the unit outward
normal on $\pOmega$, and $Bi$ is the Biot number. Here, $\Gamma_R$ is a portion of the boundary $\pOmega$ on which the inflow heat flux is $1$. The rest of the boundary is assumed to have Robin boundary condition.
We are interested in reconstructing the distributed log conductivity $\ub$,
given noisy measurements of temperature $w$ observed on $\Omega$. 

The standard $H^1\LRp{\Omega}$ finite element method is used to
discretize the misfit and the
regularization operator. The synthetic truths that we seek to recover
are a 1-D sinusoidal curve, a 2-D Gaussian on a thermal fin, and a
cube with nonzero log conductivity values on a sphere
in the center and semispheres in the opposing corners. Figure
\figref{Fullmap} shows representations of $\ub_\mathrm{truth}$ on a
mesh for these cases.

The synthetic noisy temperature observations are then generated at
all mesh points through the forward
model \eqnref{pde}. The misfit vector generated from
\figref{truth1} has data dimension $N =1025$ (with 1\% percent added noise), from
\figref{truth2} has data dimension $N = 1333$ (with .1\% percent added noise), and from
\figref{truth3}  has data dimension $N = 2474$ (with .2\% percent added noise), respectively.  


\begin{figure}[h!t!b!]
  \subfigure[truth $\ub$ for 1D experiment]{
    \includegraphics[trim=0.5cm 6.5cm 2cm 7.0cm,clip=true,width=0.47\columnwidth]{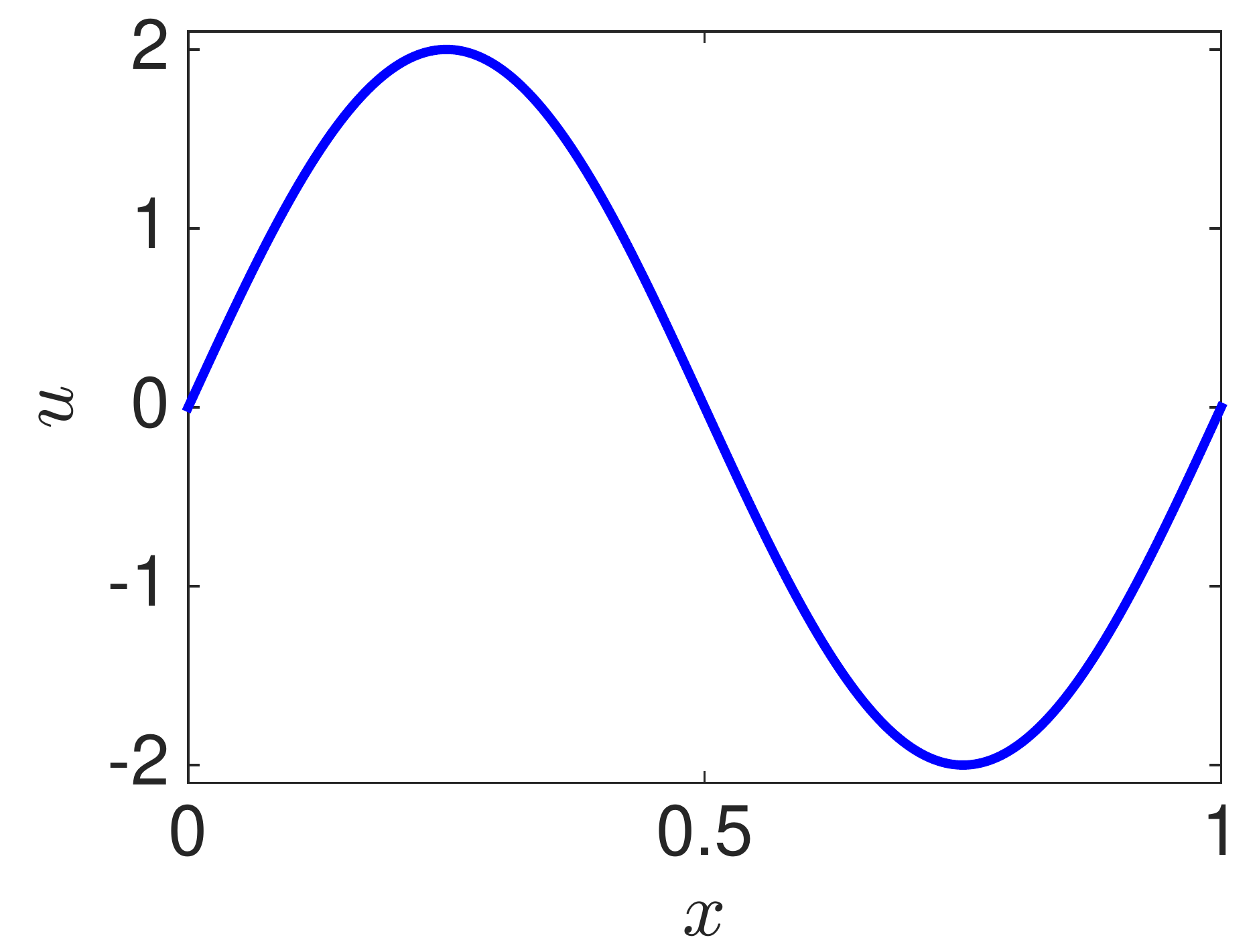}
    \figlab{truth1}
  }
\subfigure[truth $\ub$ for 2D experiment]{
    \includegraphics[trim=1.3cm 6.8cm 1.3cm 7.3cm,clip=true,width=0.47\columnwidth]{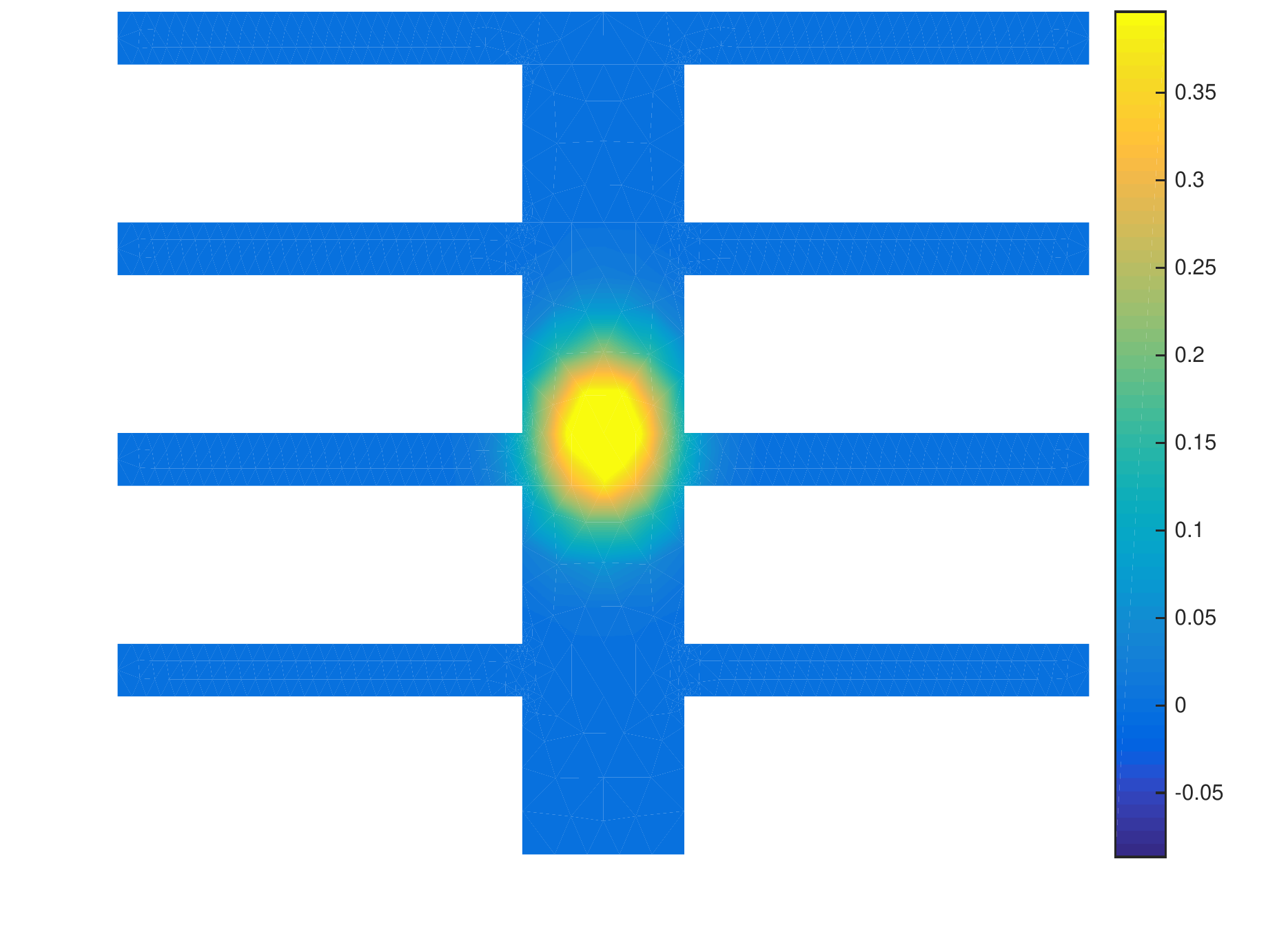}
    \figlab{truth2}
  }
\centering
\subfigure[truth $\ub$ for 3D experiment]{
    \includegraphics[trim=2.4cm 2.6cm 1.3cm 6.5cm,clip=true,
    width=0.49\columnwidth]{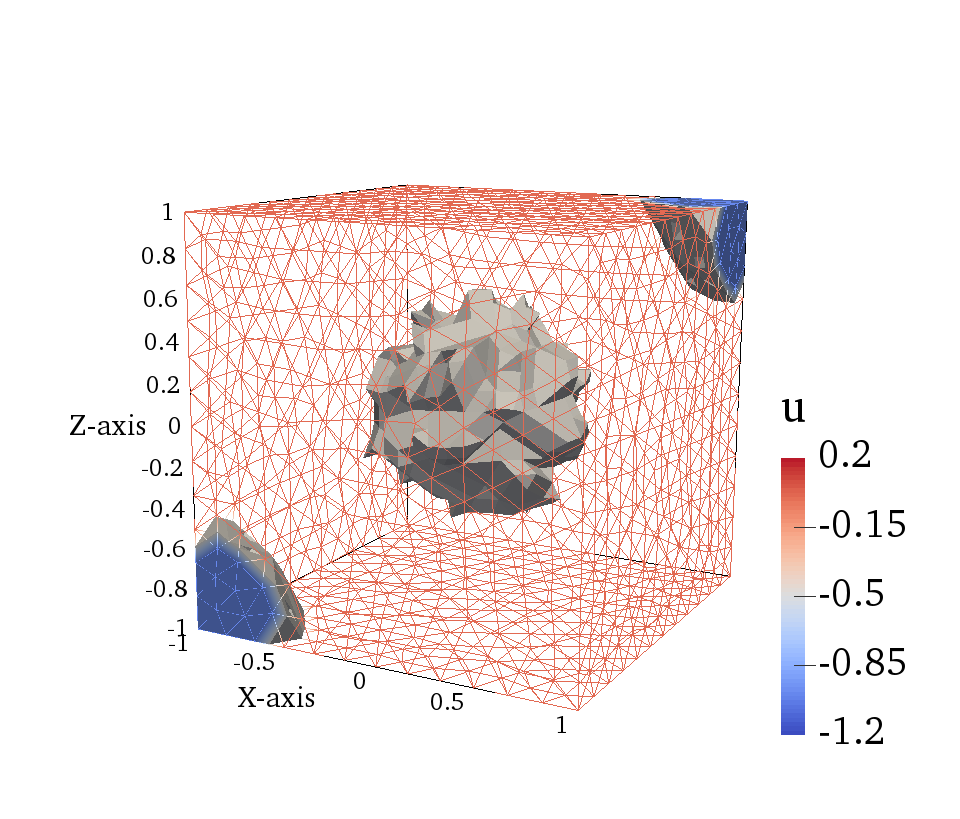}
    \figlab{truth3}
  }
  \caption{The distributed truth log conductivity parameters used in the
    experiments. The parameter fields are used to
    obtain noise-corrupted temperature data through the
    forward model \eqnref{pde}.}
  \figlab{Fullmap}
\end{figure}

For the inversion results we use an implementation of the trust region inexact
Newton conjugate gradient method, for which some of the main ideas can
be found in \cite{ColemanLi96,
  BranchColemanLi99,
  Bui-Thanh07,NocedalWright06}. Unless otherwise noted, the
stopping criteria is when the Newton step size, cost function value,
or norm of the gradient falls below $10^{-6}$. 

\subsection{Convergence results}

We first compare plots of the RMA cost $\Jn\LRp{\ub_0}$ to
the original cost $\J\LRp{\ub_0}$ for a fixed distributed parameter $\ub_0$,
using the model heat problem \eqnref{pde}. We choose a random $\ub_0$ from
the prior distribution and construct the RMA cost $\Jn\LRp{\ub_0}$
with the various random projections listed above. Since $\ub_0$ lives
in high-dimensional space $\R^m$, where $m$ is the number of finite
element nodal values, for
the purpose of visualization Figure \figref{RMA3D} shows plots of the
RMA cost $\hat{\J}_n\LRp{\kappa}: = \Jn\LRp{\ub_0 + \kappa \sb}$ in a
direction $\sb:=\Grad\J\LRp{\ub_0}$ for the 3D example. For each of
the random projections tested we observe convergence of
$\hat{J}_n\LRp{\kappa}$ to $\hat{J}\LRp{\kappa}$ as $n$  
increases. More importantly, for all distributions, the minimizer of
$\hat{J}\LRp{\kappa}$ is well-approximated by $\hat{J}_n\LRp{\kappa}$,
even for $n$ small, as shown by Theorem \theoref{morthm}. That is,
although $\J_n$ for $n=\mc{O}(1)$ is far from $\J$, the local
minimizers align. This is consistent with observed fidelity of randomized MAP
    points despite the slow convergence of the randomized cost, and
    similar phenomena seen in related methods.  Plots
with distributions other than Achlioptas and for the 1D and 2D
examples are omitted when results are similar to the 3D Achlioptas experiments (see \verb"http://users.ices.utexas.edu/~ellenle/RMAplots.pdf").
\begin{figure}[h!t!b!]
  \centering
  \includegraphics[trim=0.8cm 6.5cm 2.0cm 6.9cm,clip=true,width=0.58\columnwidth]{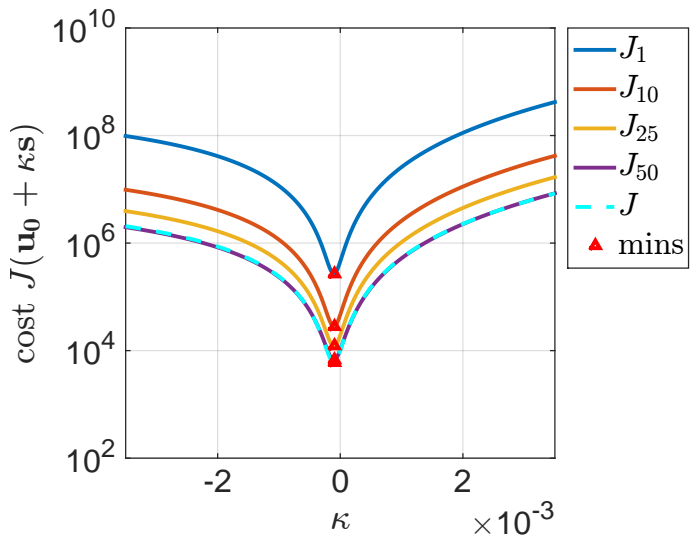}
  \caption{Contours of the RMA cost $\Jn$ with different $n$ versus
    the original cost $\J$, for the 3D example with the Achlioptas distribution. Contours are
    evaluated along a 1D direction $\sb$ parameterized by $\kappa$ and
    centered at a
    random parameter $\ub_0$ in the prior distribution. Red triangles
    indicate the minimum values of each contour. 
  }
  \figlab{RMA3D}
\end{figure}



Theorem \theoref{error} states that $\umap_n$, the minimizer of
$\Jn$,  and the minimum objective function value $\Jn^\star$
converge at the same rate, given by the distortion tolerance $\varepsilon$, but with different
constants. Figure \figref{RMA3D} illustrates how an RMA solution $\hat{\ub}^\star_n$ may converge
quickly to ${\hat\ub}^\star$,  although convergence of the minimum
value $\hat{\Jn}(\hat{\ub_n})$ to $\hat{\J}(\hat{\ub})$ can be slow due to the different constant.  To test this hypothesis at the actual minimizer $\umapn$, we plot
the error of the RMA MAP point $\umapn$ and its
corresponding optimal value $\Jn^\star$ in Figure \figref{error3D} for
the 3D example and the Achlioptas random projection\footnote[2]{Again,
  similar
  results are seen with the 1D and 2D examples and with different
  random projections. They are omitted here.}. Data shown is the
average of five runs. Both the absolute errors $|\Jn^\star-\J^\star|$ and
 $\nor{ \umapn - \umap }$ and normalized errors
 $|\Jn^\star-\J^\star|/ |\J^\star| $ and $\nor{ \umapn -
   \umap } / \nor{\umap}$ are shown, and an  $\mc{O}(1/\sqrt{n})$ reference
 curve is plotted to show the convergence rate is indeed
 $\mc{O}(1/\sqrt{n})$ for both $\umapn$ and $\Jn^\star$. However, the absolute error of
$\umapn$ is orders of magnitude smaller than $\Jn^\star$ for all
considered random projections. Also, the relative error in
$\umapn$ decreases much faster than the relative error in $\Jn^\star$
for $1 < n < 100$. Therefore a convergence analysis of the randomized
cost $\Jn$ alone is not adequate for understanding the method
efficacy in this range; the additional theory in Section
\secref{validity} is required to characterize solution accuracy for
$n$ in the range of interest.

\begin{figure}[h!t!b!]

  \subfigure[absolute convergence]{
    \includegraphics[trim=1.1cm 6.6cm 2.0cm 7.0cm,clip=true,width=0.47\columnwidth]{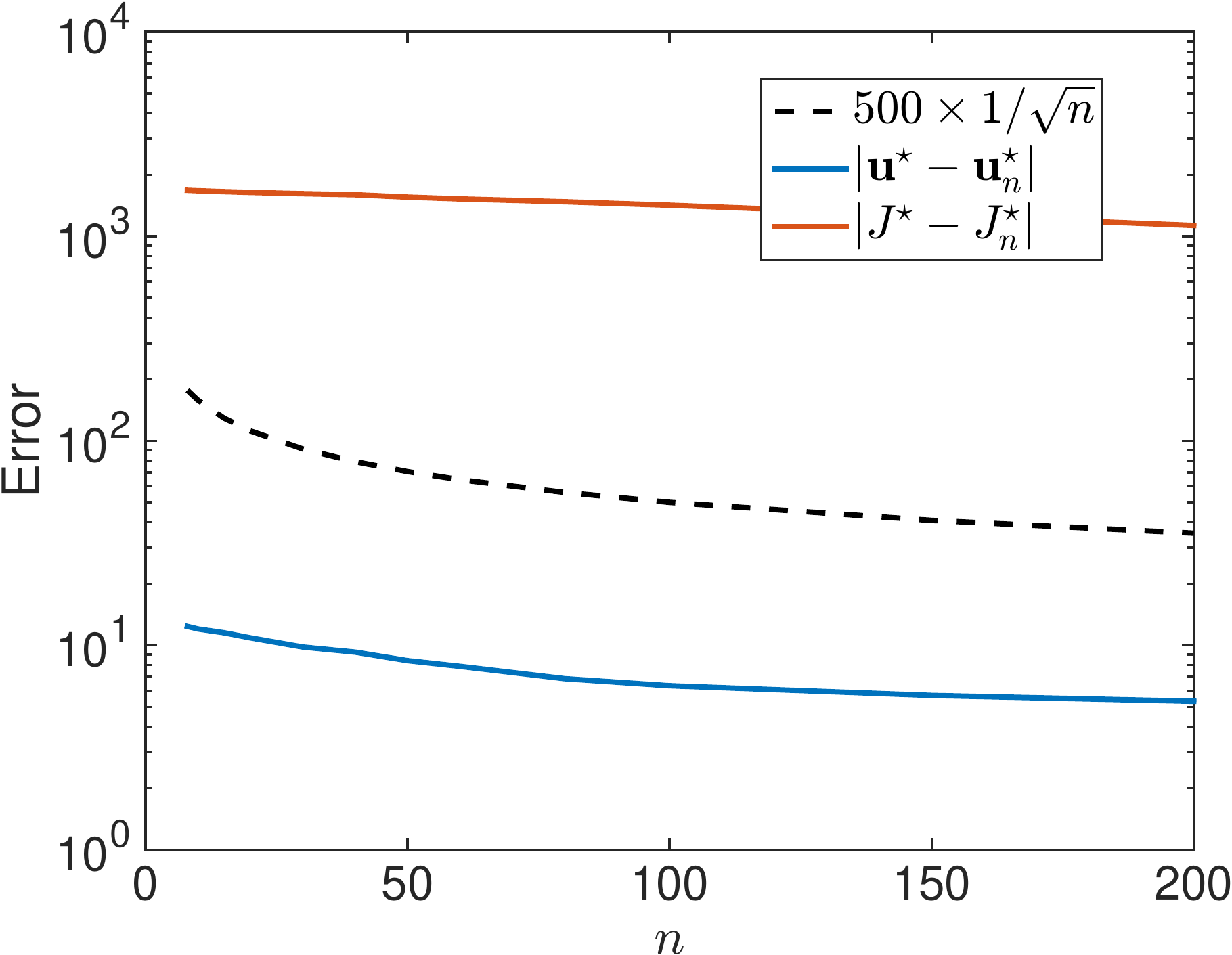}
    \figlab{absolute}
  }   
  \subfigure[relative convergence]{
    \includegraphics[trim=1.1cm 6.6cm 2.0cm 7.0cm,clip=true,width=0.47\columnwidth]{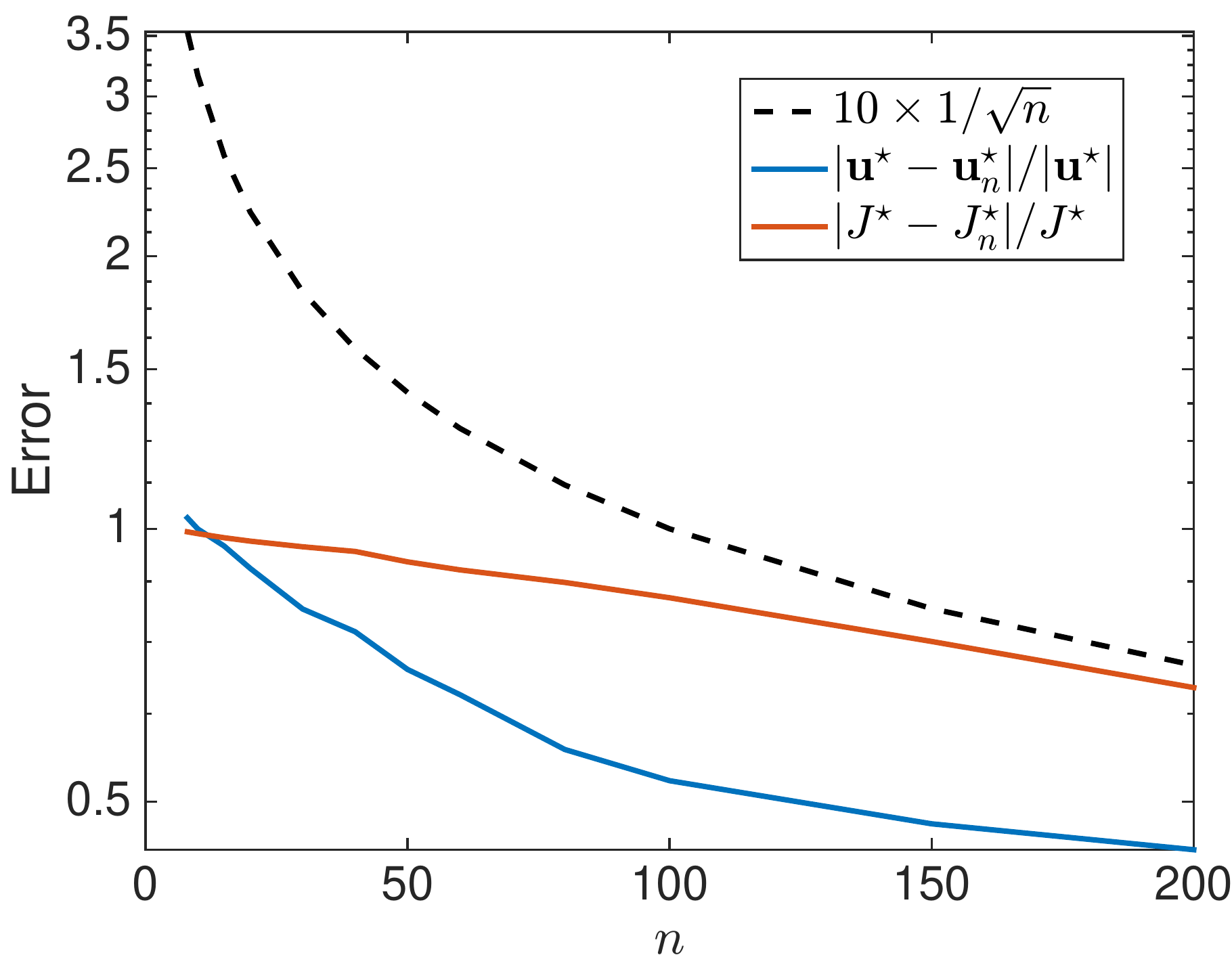}
    \figlab{relative}
  }
\caption{Log-linear plots of absolute errors $|\Jn^\star-\J^\star|$ and
 $\nor{ \umapn - \umap }$, relative errors
 $|\Jn^\star-\J^\star|/ |\J^\star| $ and $\nor{ \umapn -
   \umap } / \nor{\umap}$, and $\mc{O}(1/\sqrt{n})$ reference curves show the  $\mc{O}(1/\sqrt{n})$ convergence rate for both $\umapn$ and $\Jn^\star$ as given by
  Theorem \theoref{error}.}
  \figlab{error3D}
\end{figure} 



Inversion results from minimizing the RMA cost with different $n$ in the 1D, 2D, and 3D example are shown alongside the true MAP
estimate $\umap$ in Figures \figref{1Dinversion}, 
\figref{2Dinversion} and \figref{3Dinversionparaview}. The figures shown
are results with  $\rb$ distributed by the Achlioptas
distribution ($66\%$ sparse). We see that the original MAP point $\umap$ is
well-approximated by the RMA solution $\umapn$ in all cases with $50 \leq n \leq 100$. 

\begin{figure}[h!t!b!]

  \subfigure[$n = 5$]{
    \includegraphics[trim=2.0cm 6.7cm 2.5cm 7.1cm,clip=true,width=0.3\columnwidth]{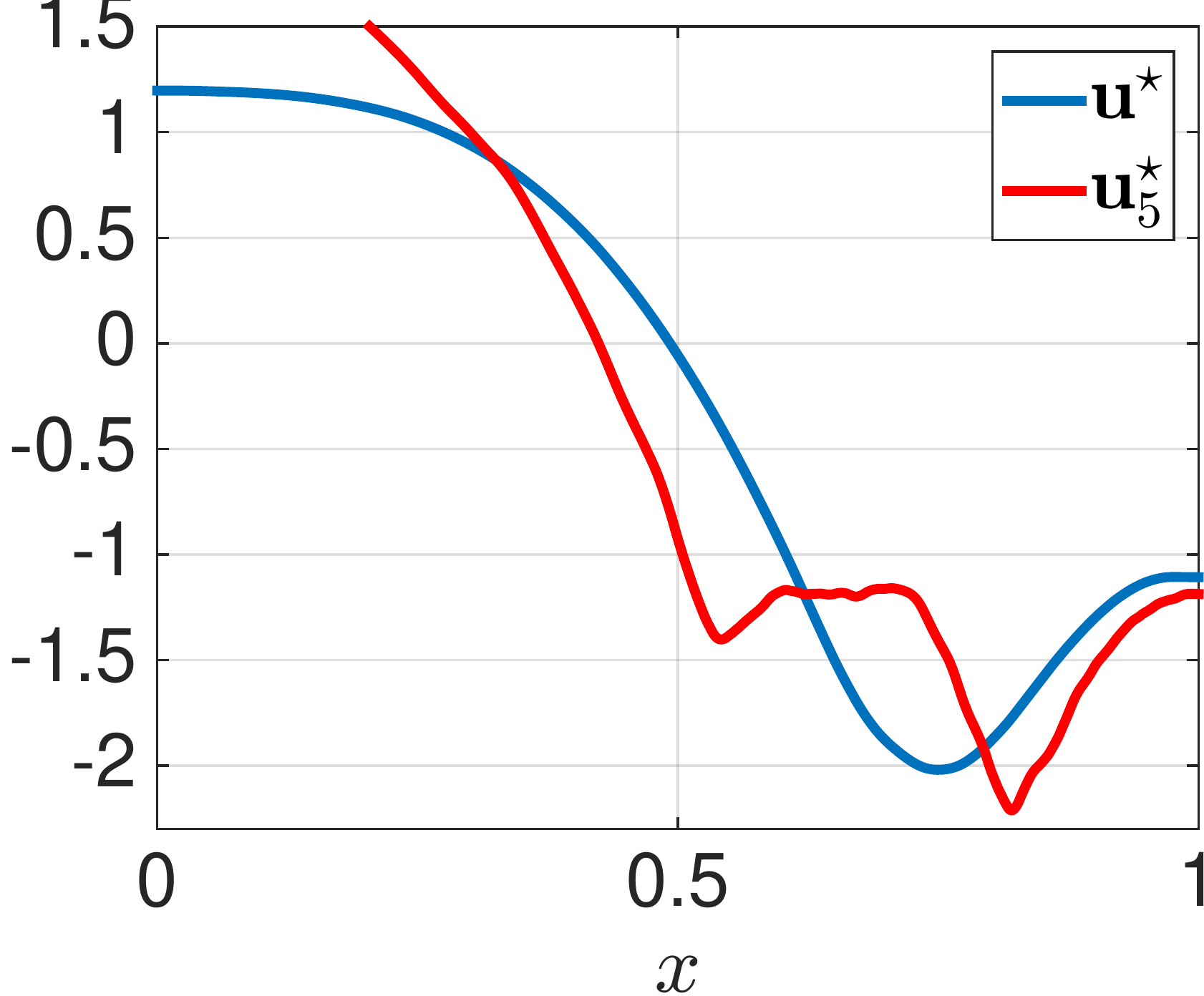}
    \figlab{Achli1D5}
  }   
  \subfigure[$n = 10$]{
    \includegraphics[trim=2.0cm 6.5cm 2.5cm 7.1cm,clip=true,width=0.3\columnwidth]{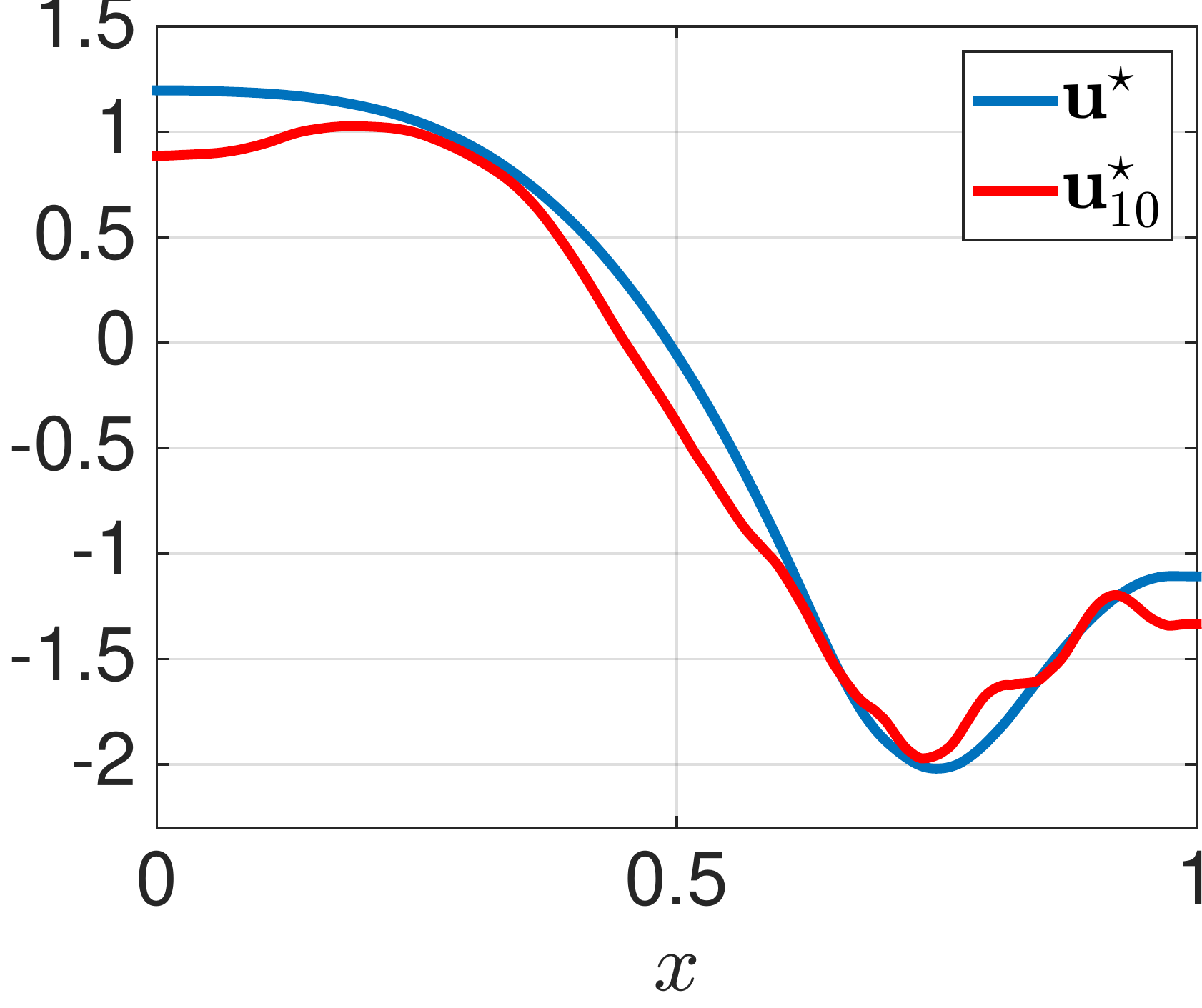}
    \figlab{Achli1D10}
  }
  \subfigure[$n = 50$]{
    \includegraphics[trim=2.0cm 6.5cm 2.5cm 7.1cm,clip=true,width=0.3\columnwidth]{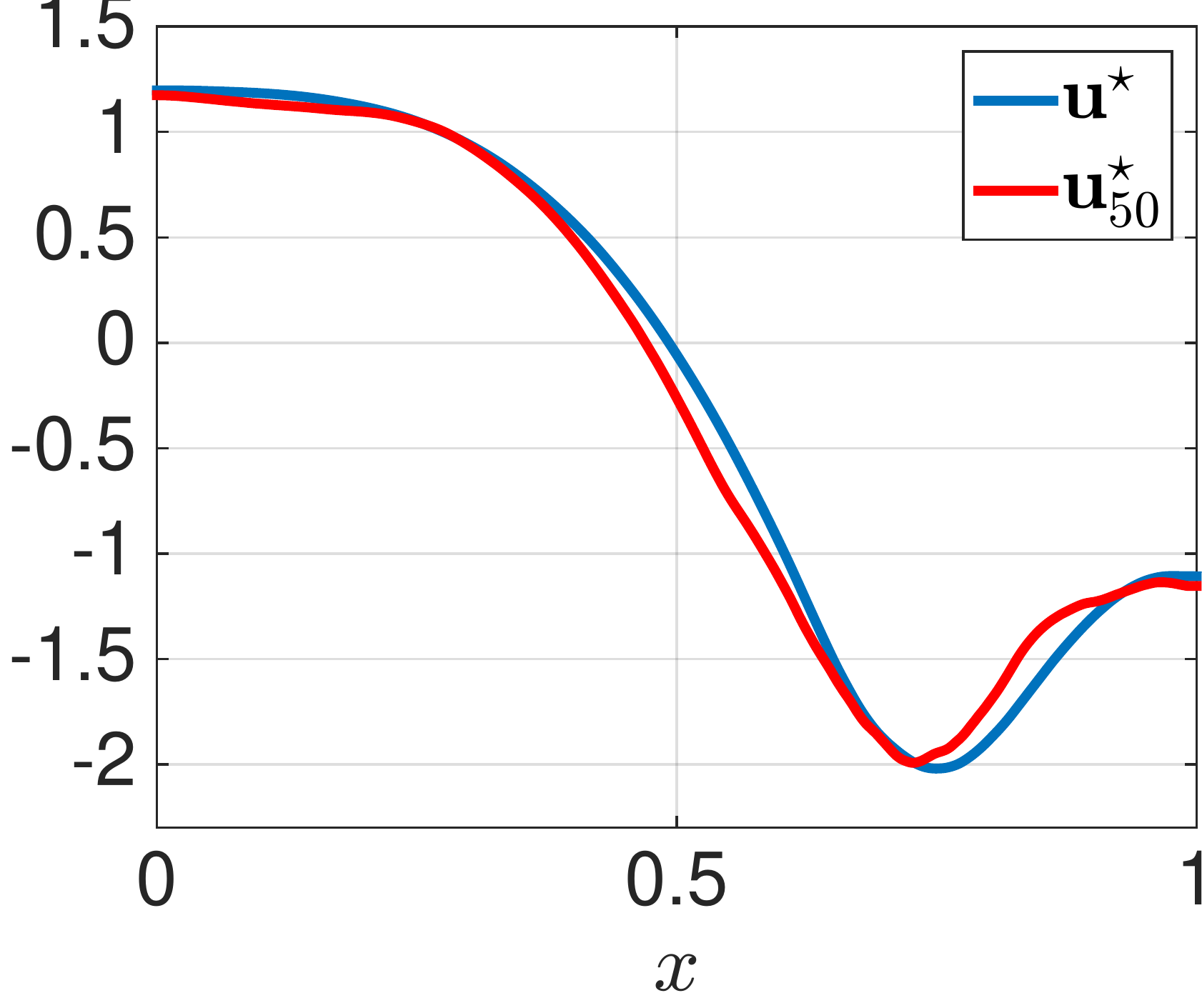}
    \figlab{Achli1D50}
  }
  \subfigure[$n = 100$]{
    \includegraphics[trim=2.0cm 6.5cm 2.5cm 7.1cm,clip=true,width=0.31\columnwidth]{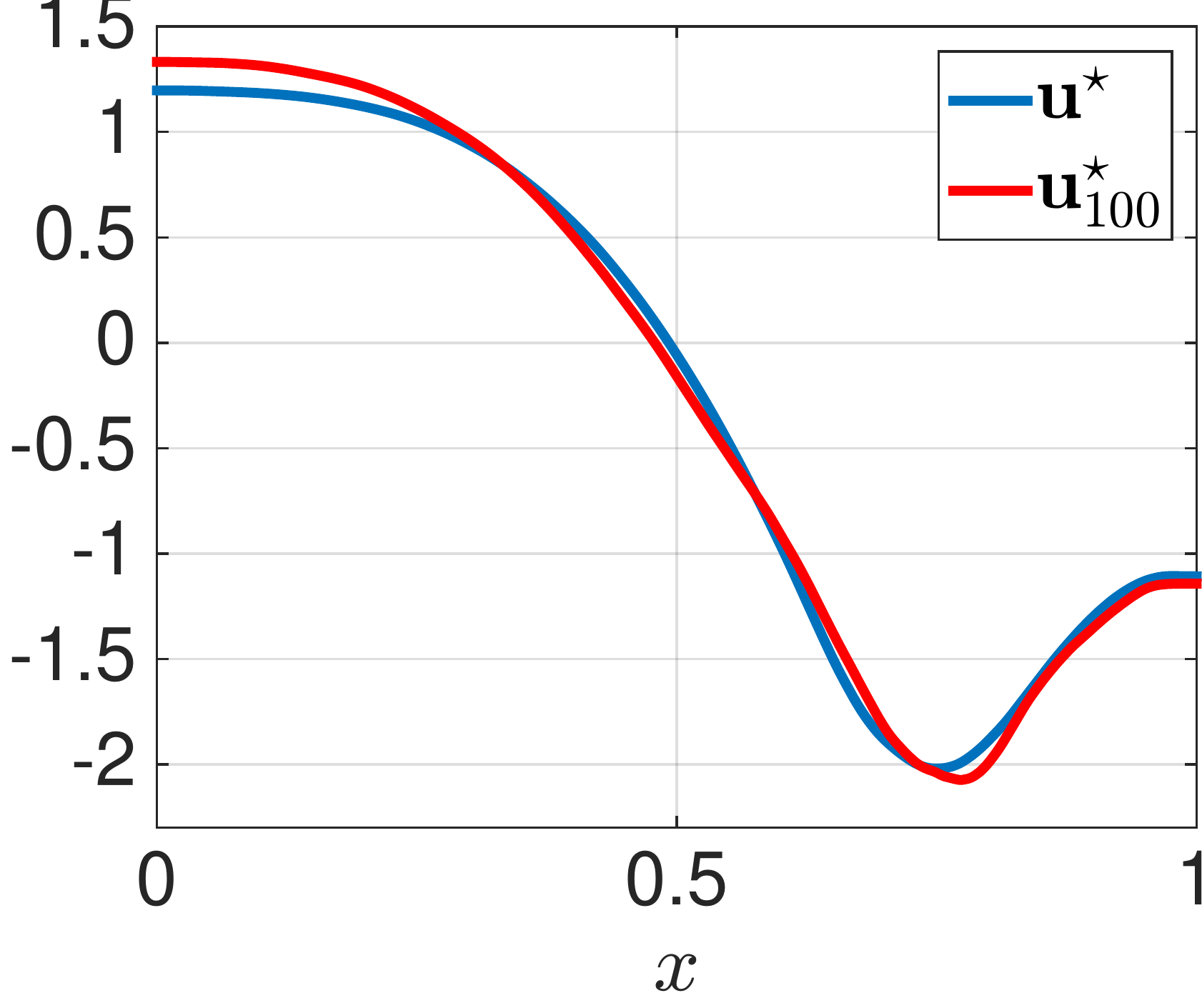}
    \figlab{Achli1D100}
  }
  \subfigure[$n = 300$]{
    \includegraphics[trim=2.0cm 6.5cm 2.5cm 7.1cm,clip=true,width=0.31\columnwidth]{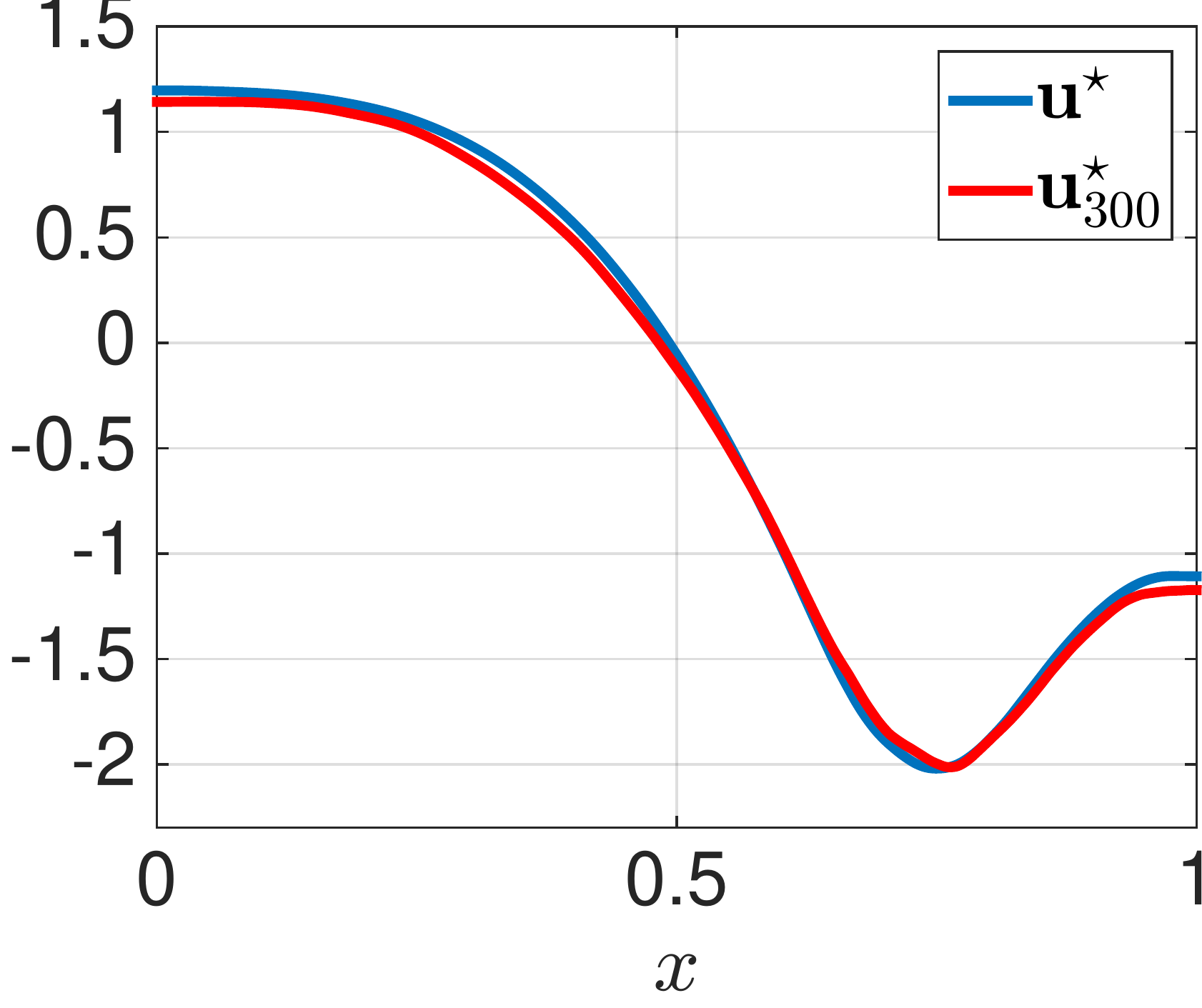}
    \figlab{Achli1D300}
  }
  \subfigure[$n = 600$]{
    \includegraphics[trim=2.0cm 6.5cm 2.5cm 7.1cm,clip=true,width=0.31\columnwidth]{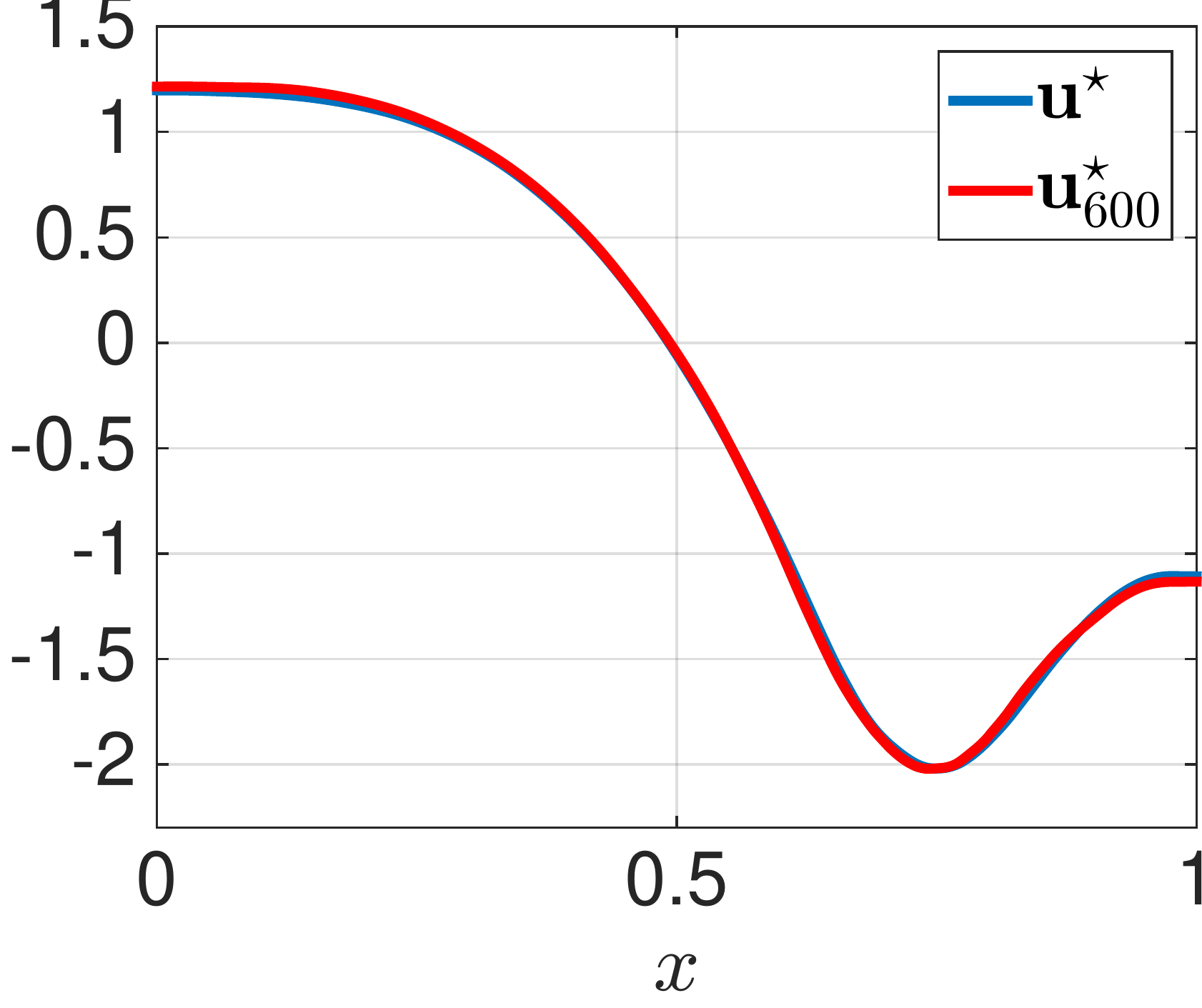}
    \figlab{Achli1D1}
  }
  \caption{1D elliptic PDE example: Convergence of
    $\umapn$ to $\umap$ as $n$ increases. The Achlioptas
    random projection ($66\%$ sparse) is used for $\Sigb$ and the
    original data dimension is $N=1025$.} 
  \figlab{1Dinversion}
\end{figure} 

\begin{figure}[h!t!b!]
  \subfigure[$n=30$]{  
    \includegraphics[trim=1.3cm 6.8cm 1.3cm 7.3cm,clip=true,width=0.48\columnwidth]{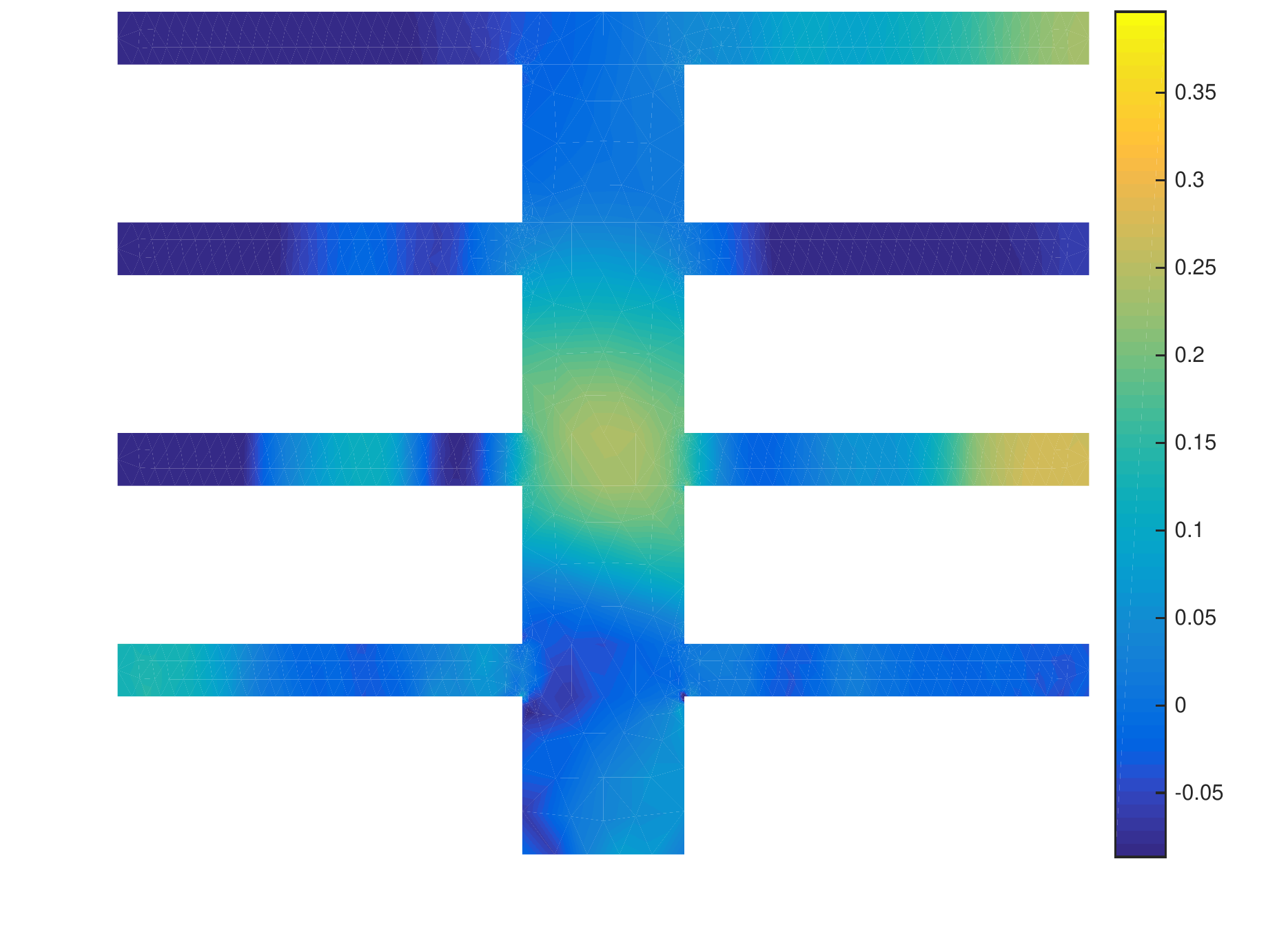}
  }
  \subfigure[$n=50$]{  
    \includegraphics[trim=1.3cm 6.8cm 1.3cm 7.3cm,clip=true,width=0.48\columnwidth]{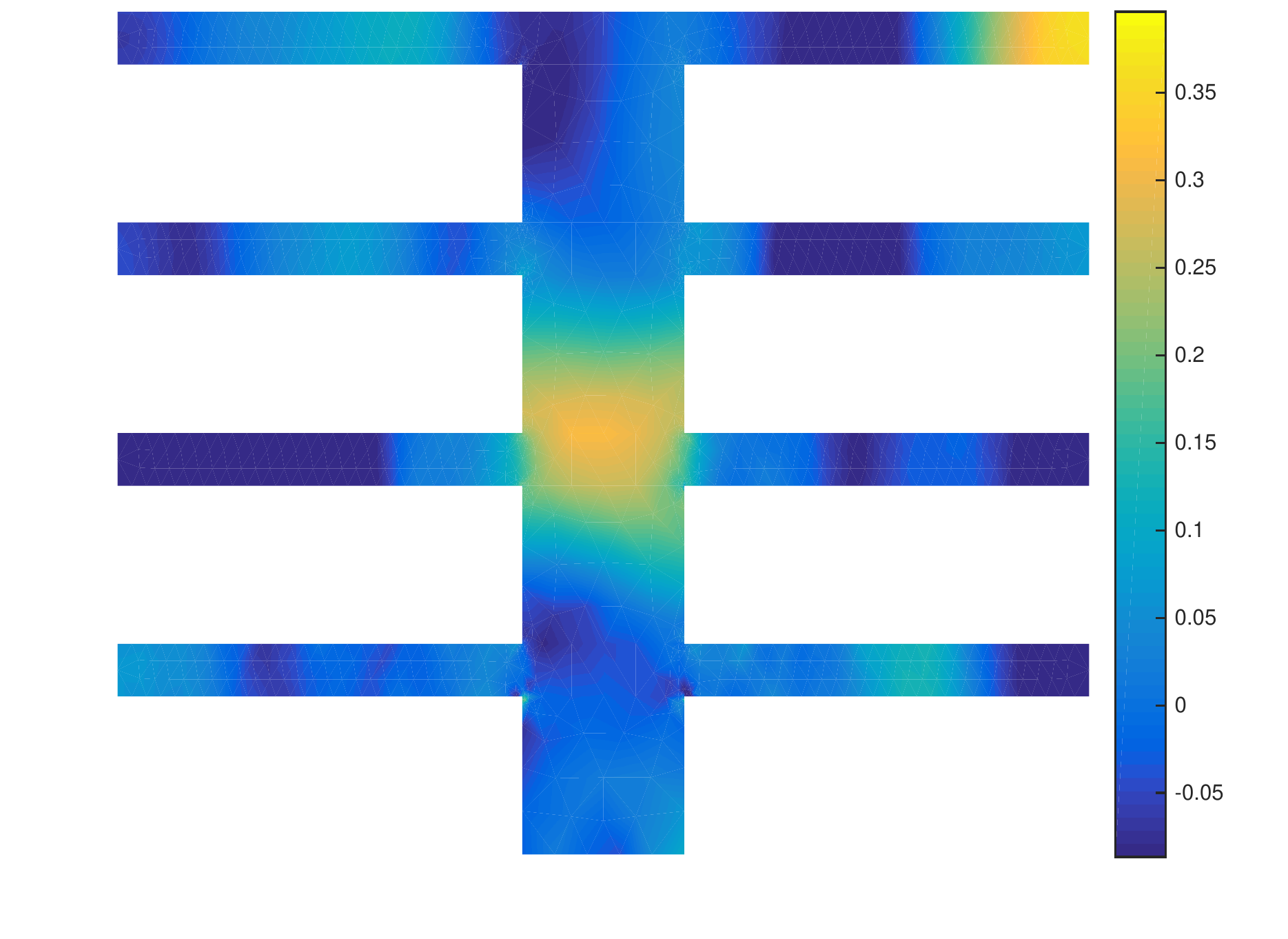}
  }
  \subfigure[$n=100$]{  
    \includegraphics[trim=1.3cm 6.8cm 1.3cm 7.3cm,clip=true,width=0.48\columnwidth]{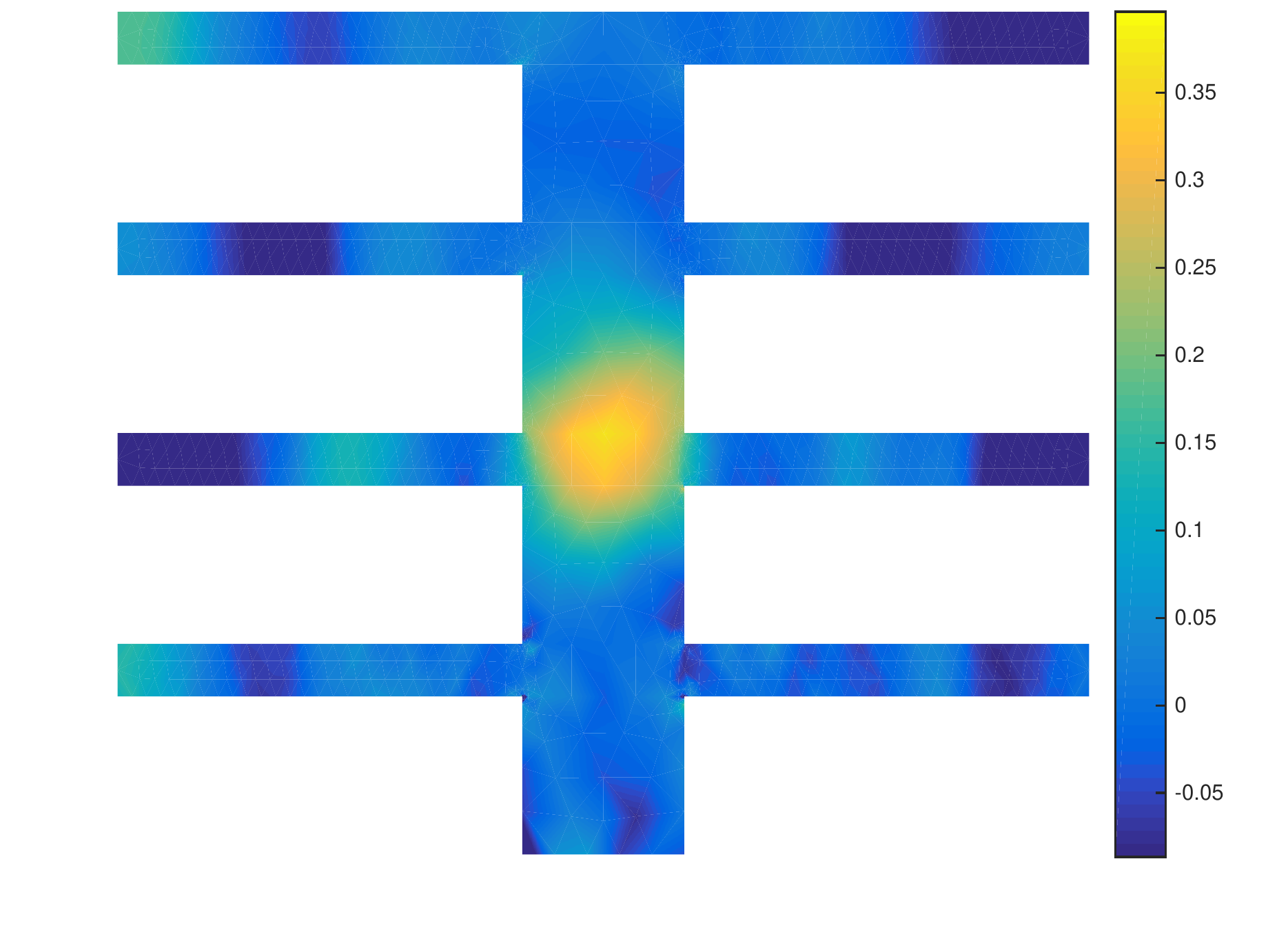}
  }
  \subfigure[$\umap$]{  
    \includegraphics[trim=1.3cm 6.8cm 1.3cm 7.3cm,clip=true,width=0.48\columnwidth]{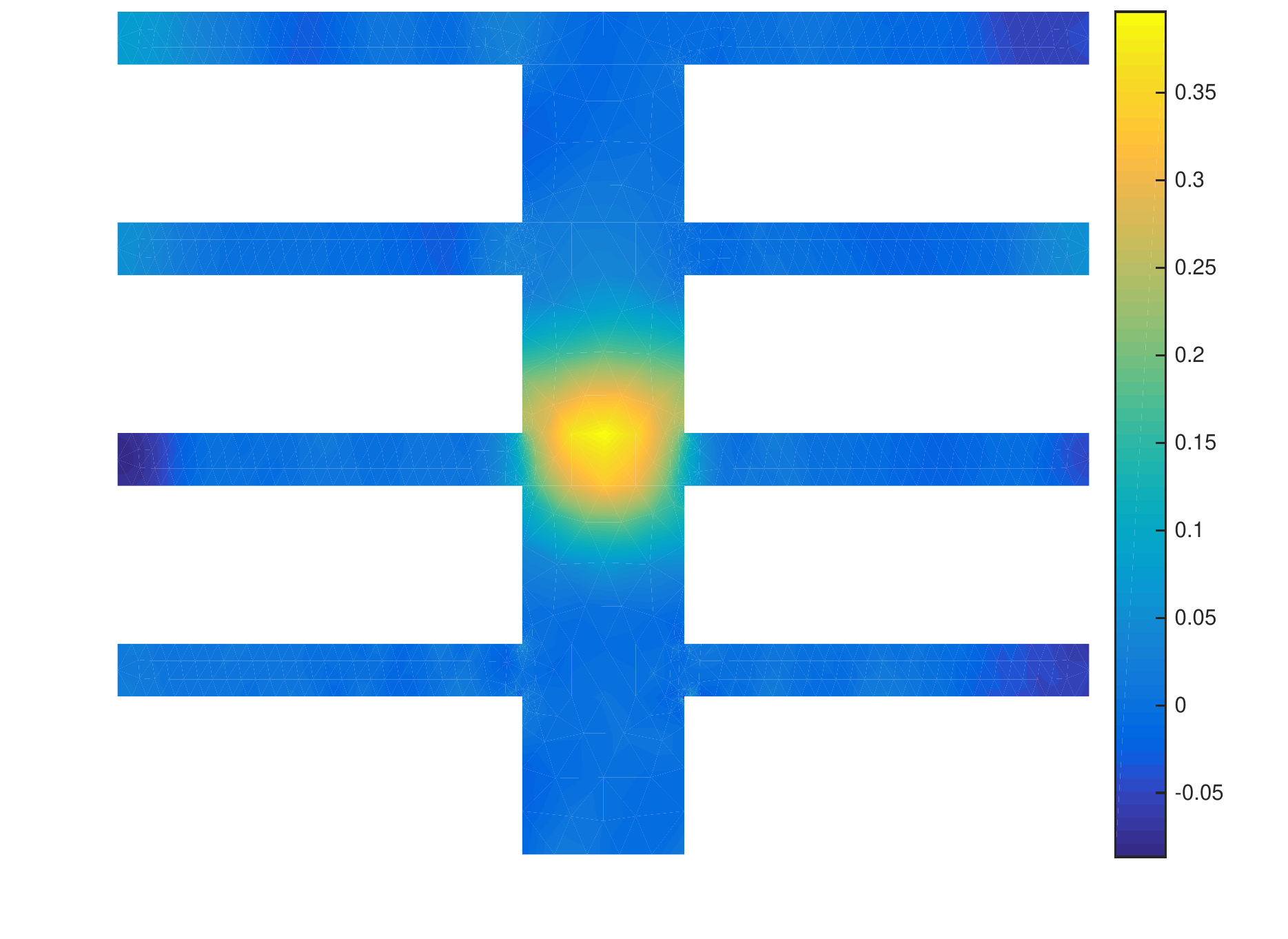}
  }

  \caption{2D elliptic PDE example: Convergence of
    $\umapn$ to $\umap$ as $n$ increases. The Achlioptas
    random projection ($66\%$ sparse) is used for $\Sigb$ and the
    original data dimension is $N=1333$.} 
  \figlab{2Dinversion}
\end{figure}

\begin{figure}[h!t!b!]
  \subfigure[$n=50$]{
    \includegraphics[trim=2.4cm 2.6cm 1.3cm 6.5cm,clip=true,width=0.49\columnwidth]{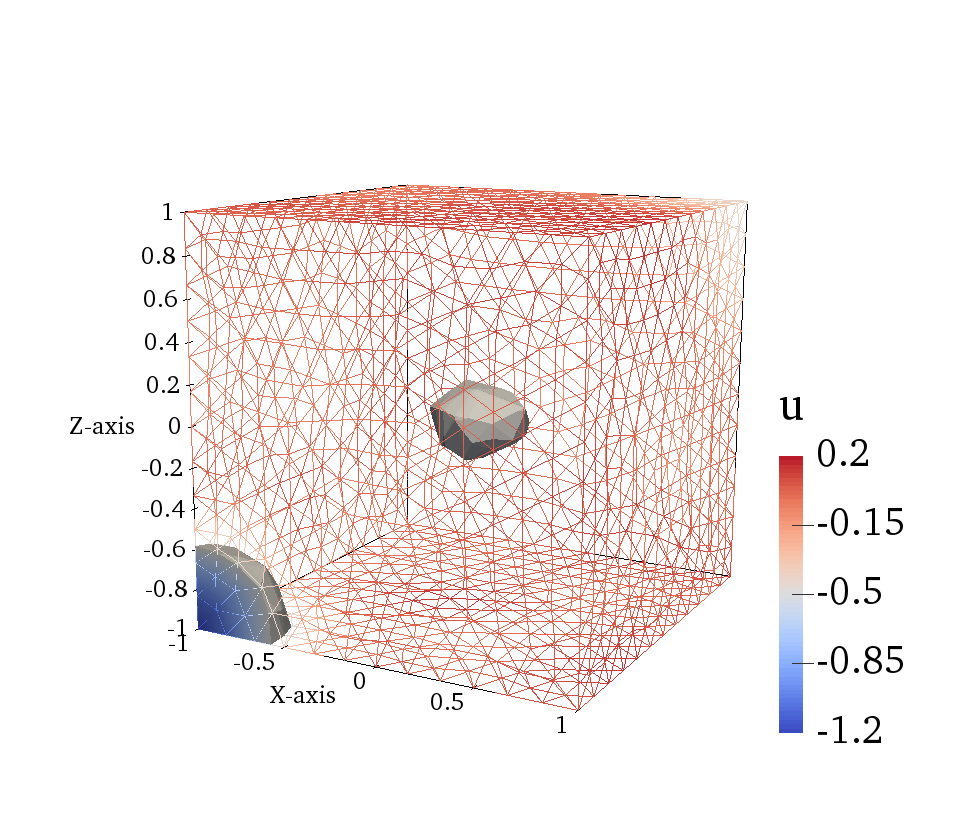}
    \figlab{Elliptic3Dachlioptas50}
  }
  \subfigure[$n=100$]{
    \includegraphics[trim=2.4cm 2.6cm 1.3cm 6.5cm,clip=true, width=0.49\columnwidth]{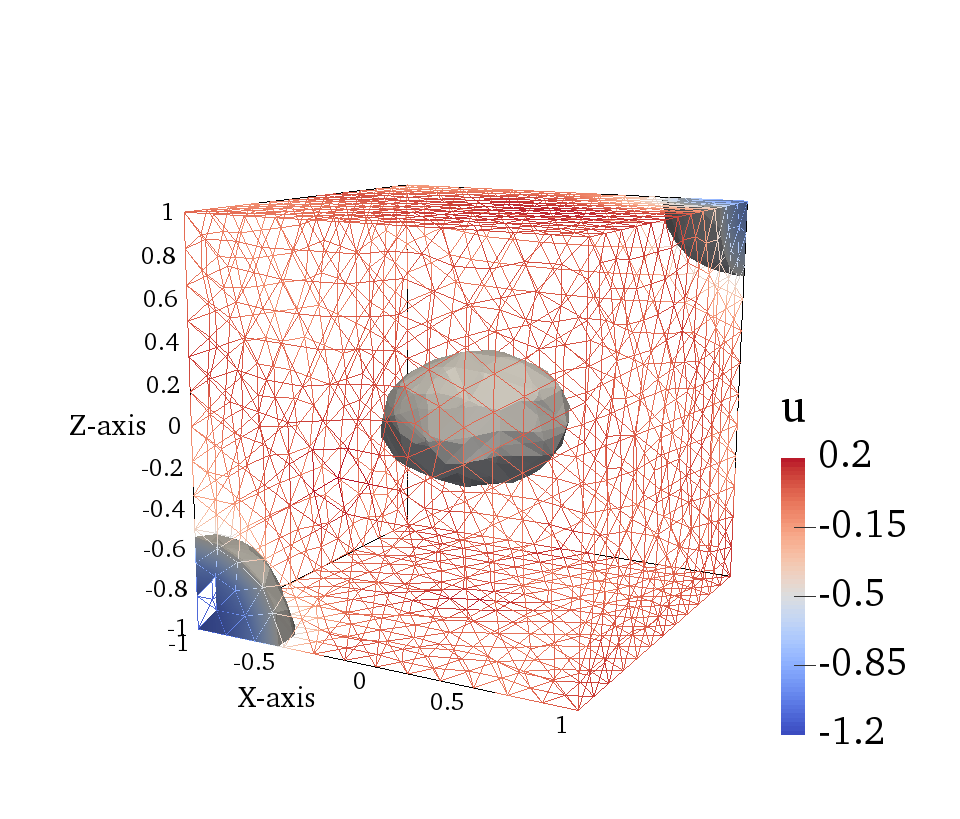}
    \figlab{Elliptic3Dachlioptas100}  
  }
  \subfigure[$n=150$]{
    \includegraphics[trim=2.4cm 2.6cm 1.3cm 6.5cm,clip=true, width=0.49\columnwidth]{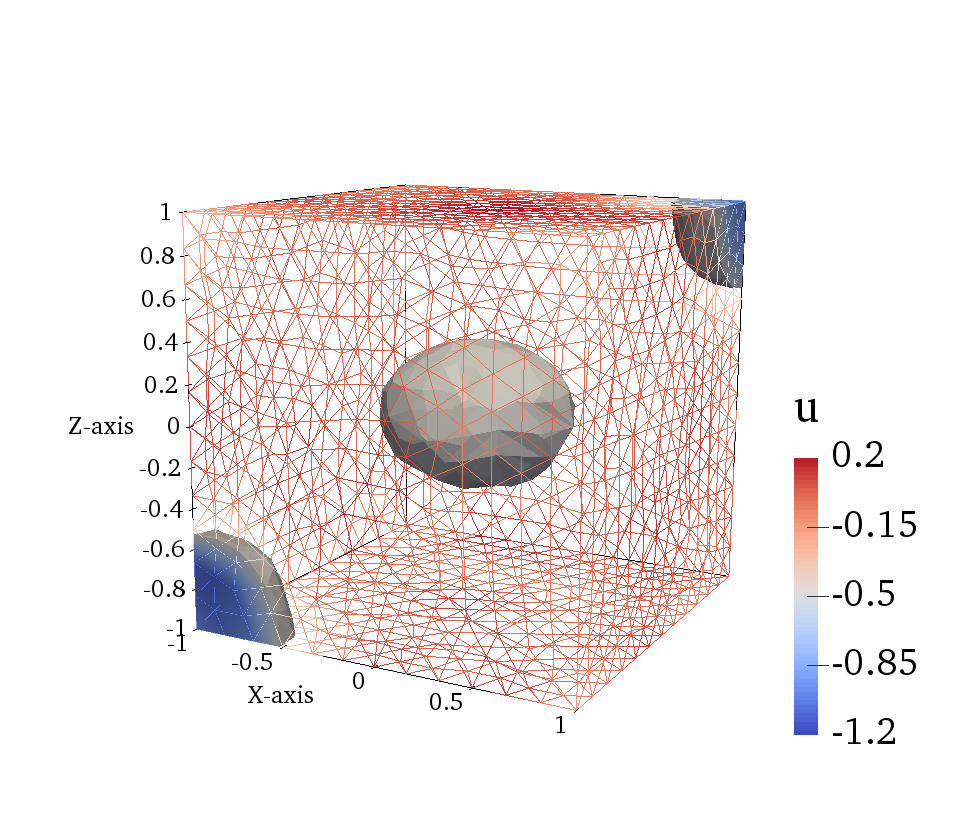}
    \figlab{Elliptic3Dachlioptas150}
  }
  \subfigure[$\umap$]{
    \includegraphics[trim=2.4cm 2.6cm 1.3cm 6.5cm,clip=true, width=0.49\columnwidth]{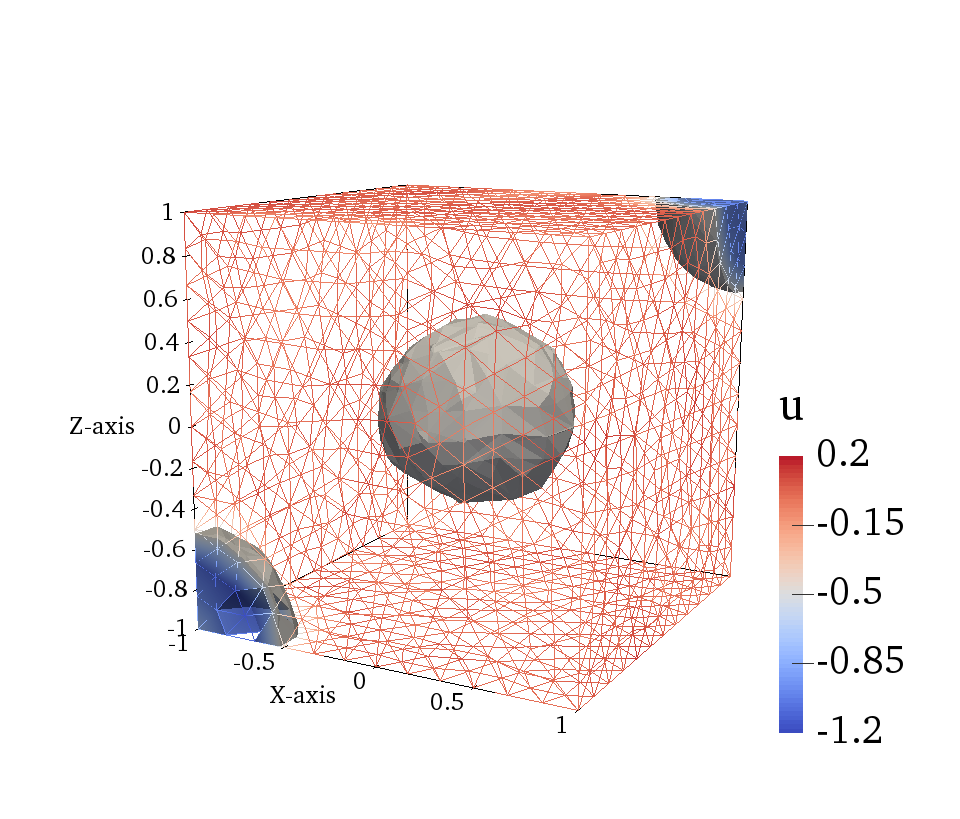} 
    \figlab{Elliptic3DFullMap}
  }
  \caption{3D elliptic PDE example: Convergence of
    $\umapn$ to $\umap$ as $n$ increases. The Achlioptas
    random projection ($66\%$ sparse) is used for $\Sigb$ and the
    original data dimension is $N=2474$.} 
    \figlab{3Dinversionparaview}
\end{figure}

In a different experiment, we consider a 3D example in which only surface observations are
available. The parameters are the same as the problem represented by Figure \figref{truth3}
but the data are now obtained from $901$ observations on the
surface of the cube (except the bottom surface), and the truth log
conductivity is nonzero within the sphere of radius $0.5$ centered at
the origin as seen in Figure \figref{3DSyntheticSurface}.
Figure \figref{3DgamMapSurface} shows the original MAP
estimate $\umap$. 
Compared to the above example the recovery is poorer, but this is
expected due to having less observational data. 
Our interest however is in reducing the computational burden caused by the large data dimension while
recovering a reasonable MAP estimation.
Subsequently, we compare the RMA MAP point $\umapn$ to
the true MAP point $\umap$ (a minimizer of $\J$).
The results in Figure \figref{3DinversionparaviewsurfaceN} show the
RMA solutions $\umapn$ as $n$ increases.
As can be seen, with $n=150$, i.e. a 6-fold reduction in the data
misfit dimension, the RMA approximation
$\umap_{150}$ is still a good approximation to the original MAP solution $\umap$.

\begin{figure}[h!t!b!]
\centering
    \includegraphics[trim=2.4cm 2.6cm 1.3cm 6.5cm,clip=true, 
    width=0.45\columnwidth]{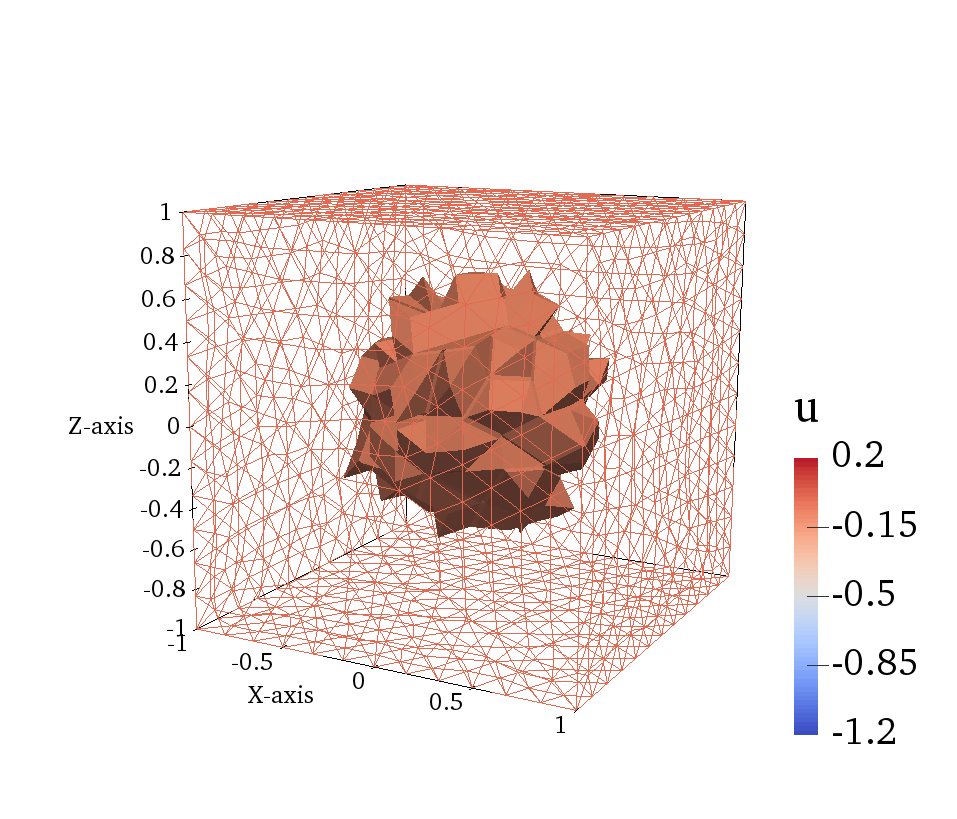}
    \caption{Truth $\ub$ for 3D experiment with surface observations: The same number of mesh
      elements as in Figure \figref{truth3} is used but now the
      synthetic parameter is a single sphere, and observational data is
      obtained from $N=901$ mesh points on the top and side surfaces of
      the cube.}
    \figlab{3DSyntheticSurface}   
  \end{figure}

  \begin{figure}[h!t!b!]
    \subfigure[$n=10$]{
      \includegraphics[trim=2.4cm 2.6cm 1.3cm 6.5cm,clip=true,
      width=0.49\columnwidth]{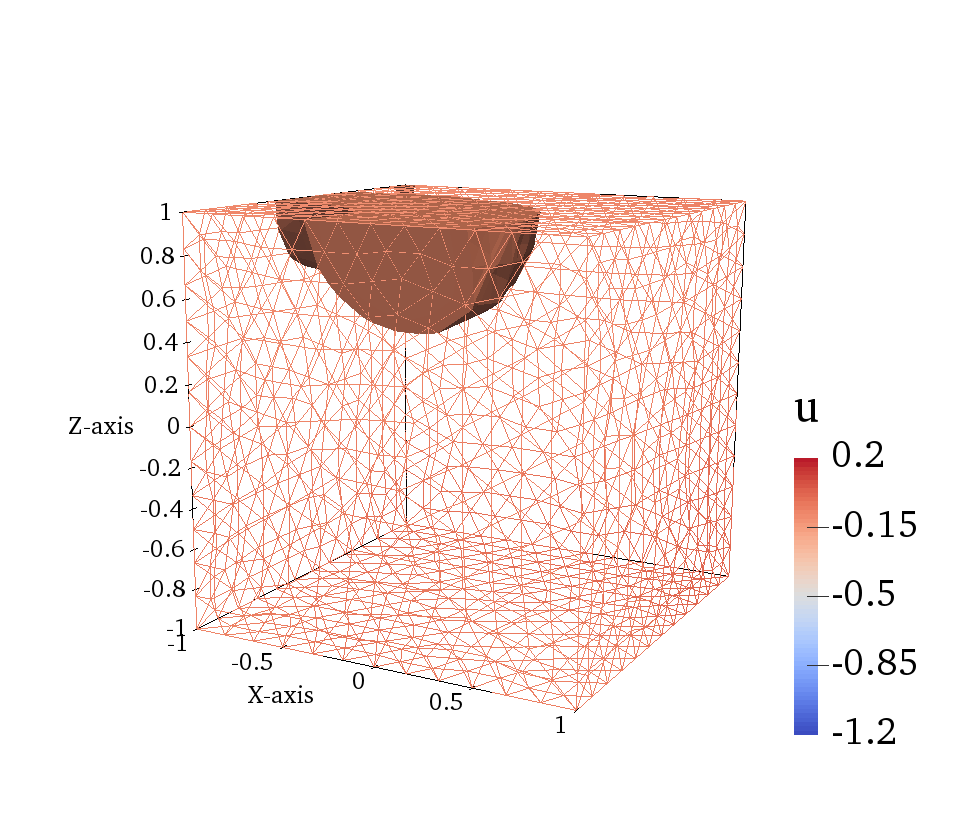}
    }
    \subfigure[$n=50$]{
      \includegraphics[trim=2.4cm 2.6cm 1.3cm 6.5cm
      ,clip=true, width=0.49\columnwidth]{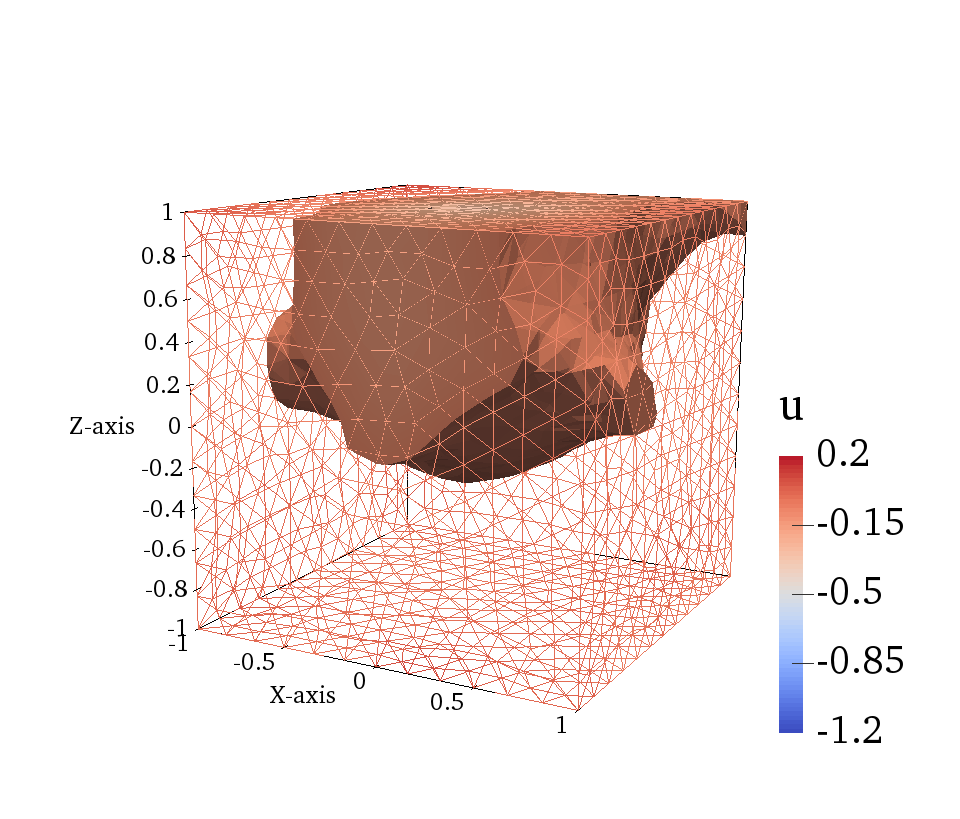}
    }
    \subfigure[$n=150$]{
      \includegraphics[trim=2.4cm 2.6cm 1.3cm 6.5cm
      ,clip=true, width=0.49\columnwidth]{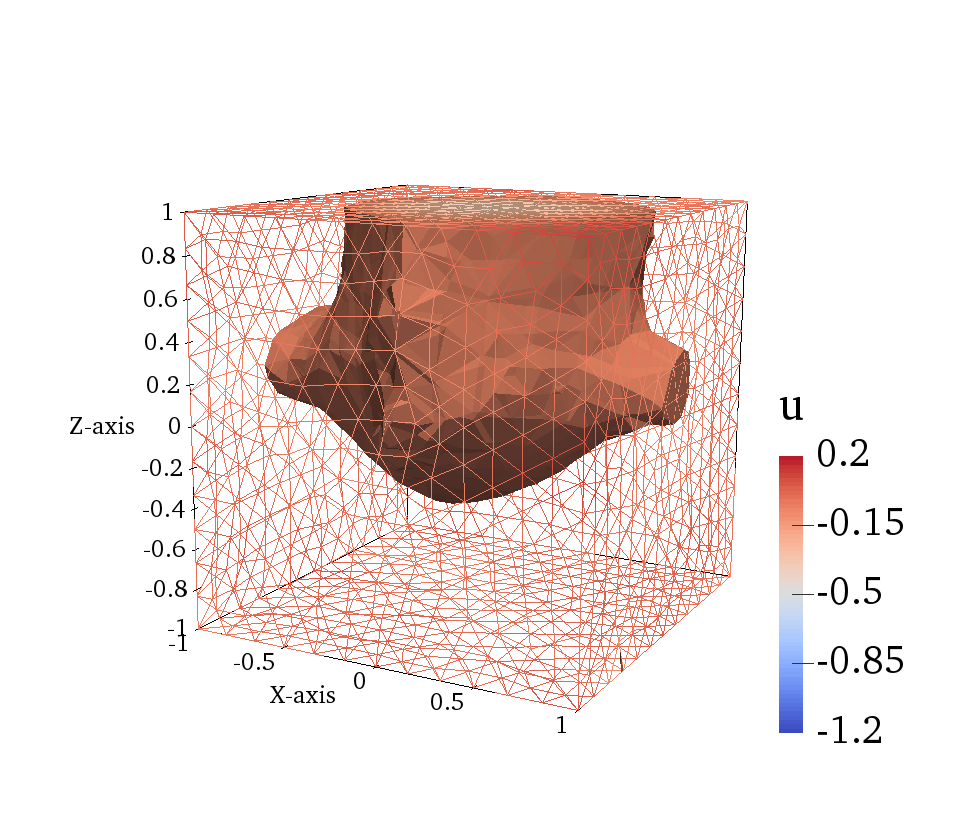}
    }
    \subfigure[$\ub^\star$ (solution of the full cost)]{
      \includegraphics[trim=2.4cm 2.6cm 1.3cm 6.5cm
      ,clip=true, width=0.49\columnwidth]{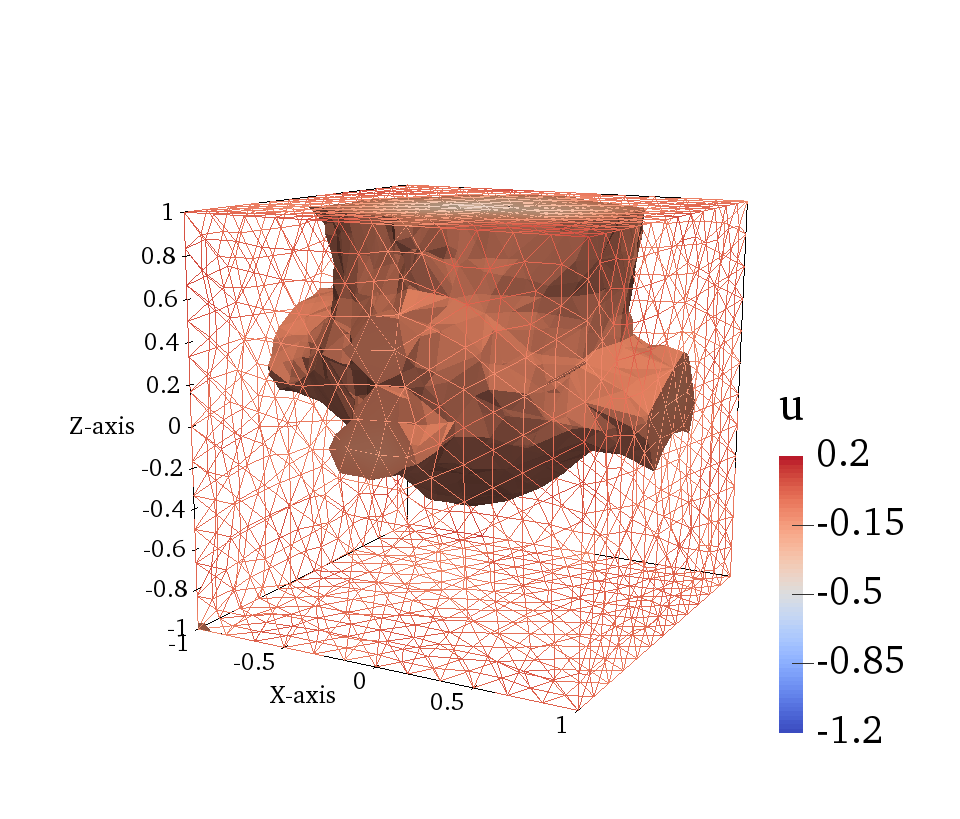}
      \figlab{3DgamMapSurface}
    }
    \caption{3D elliptic PDE example with surface observations: Convergence of $\umapn$ to $\umap$
      as $n$ increases. The MAP solution is nearly approximated 
      with an RMA reduced misfit dimension of $n=150$ (a 6-fold
      reduction from the $N=901$ observational data points on the
      surface). The Achlioptas random projection ($66\%$ sparse) is used for $\Sigb$.} 
    \figlab{3DinversionparaviewsurfaceN}
  \end{figure}

\subsection{Verification of Theorem \theoref{morthm}}

Table \tabref{morozov} presents results for solving the model problem
for the 1D, 2D, and 3D examples with Morozov's
criterion, again using the Achlioptas random projection in the
randomized misfit approach. We perform several numerical experiments
and choose an $n$ for each
example such that Morozov's principle is met for
$\Jn\LRp{\umapn}$ with $\tau' \approx 1$. We then compute the
corresponding ranges for $\tau$ that are guaranteed with probability at least $p \geq
1-e^{-\beta}$, after choosing an acceptable cost distortion tolerance of
$\varepsilon = 0.5$ and $\beta$ as large as possible from
\eqnref{Nlb}. As can be seen, evaluating $\mc{J}\LRp{\umapn}$ gives a $\tau$ within the specified range, which satisfies
Morozov's criterion. That is, even for moderately small values of $n$,
if the discrepancy principle is satisfied for
$\Jn\LRp{\umapn}$, then the discrepancy principle is also satisfied
for $\J\LRp{\umapn}$. Thus $\umapn$ is a discrepancy
principle-satisfying solution for both the randomized reduced misfit
dimension problem \eqnref{randprob} and the original problem \eqnref{MAP}.

\begin{table}[h!t!b!]
  \caption{Verification of Morozov's discrepancy principle for the RMA
    solution with $\varepsilon=0.5$.}  \tablab{morozov}
  \begin{center}
    \begin{tabular}{r|r|r|r|r|r|r|r|r}

      & $N$ & $n$ & $\mc{J}_n\LRp{\umapn}$ & $\tau '$ & $\LRs{\frac{\tau'}{1+\varepsilon},\frac{\tau'}{1-\varepsilon}}$  &$p$&$\mc{J}\LRp{\umapn}$ &  $\tau$\\
      \hline
      \hline
      1D&$1025$ & $100$ & $1220$ & $1.190$  & $\LRs{0.793,2.380}$ &$95.6\%$ &$1074$ & $1.048$  \\
      \hline
      2D &$1333$ & $50$ & $1240$& $0.930$ & $\LRs{0.620, 1.860}$ &$79.0\%$ & $1406$  & $1.055$  \\
      \hline
     3D &$2474$ & $75$ & $2646$ & $1.070$  & $\LRs{0.713,2.139}$ &$90.4\%$ & $3928$  & $1.588$  \\
      \hline
    \end{tabular}
  \end{center}
\end{table}

\subsection{Scalability and performance}
\seclab{scalability}

We study the effect of the RMA reduced misfit dimension $n$ on
the overall algorithmic scalability of solving large-scale
PDE-constrained inverse problems with high observational data
dimensions. Specifically, we wish to show that RMA convergence is independent of $r$, the level of parameter information from the
data (see Section \secref{costanalysis}). Figure \figref{zoom}
compares singular values of the prior-preconditioned
misfit Hessian $\mathbf{H}$ corresponding to the original problem cost $\J$ to the
singular values of the surrogate prior-preconditioned
misfit Hessian $\mathbf{\tilde{H}}$ corresponding to the surrogate RMA
cost $\Jn$ for $n=30,50,$ and $100$. The Hessians are each evaluated at the same random point chosen from the prior. 
Note that the RMA reduced misfit dimension $n$ is a hard upper bound
on the numerical rank of $\mathbf{\tilde{H}}$, where numerical rank is
the number of singular values greater than some threshold $\epsilon
\leq 1$. Note also the faster spectral decay of the singular values of 
$\mathbf{\tilde{H}}$ compared to $\mathbf{H}$. Faster decay demonstrates that the
action of $\mathbf{\tilde{H}}$ on a vector can be captured with fewer
modes than the action of $\mathbf{H}$, resulting in decreased
overall work complexity as detailed in Section
\secref{costanalysis}. Similar behavior is observed when the
Hessians are evaluated at zero, at another random point, and at the full MAP point $\umap$,
thus the plots are omitted.

\begin{figure}[h!t!b!]
\centering
    \includegraphics[trim=1.4cm 6.6cm 2.3cm 6.5cm,clip=true, 
    width=0.7\textwidth]{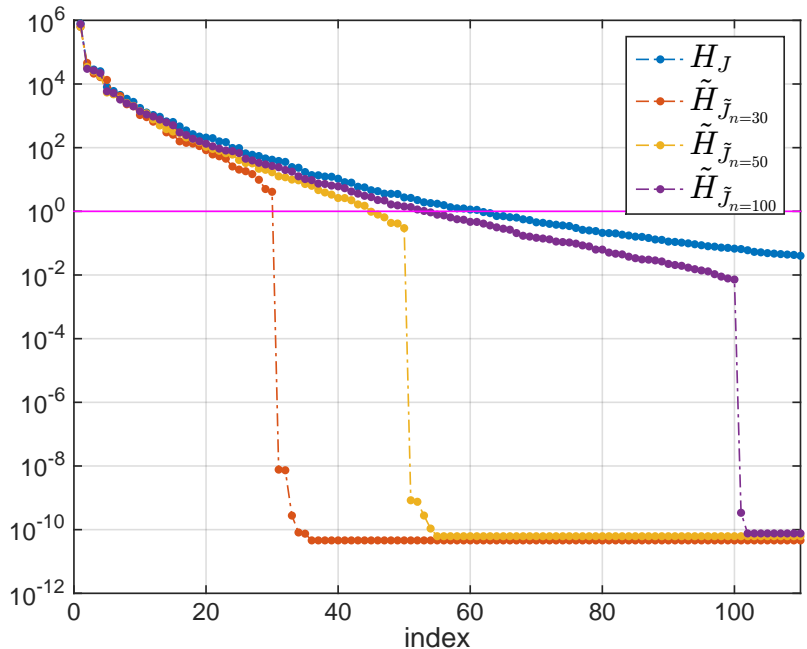}
  \caption{Log-linear spectra of the prior-preconditioned misfit Hessian for the 2D
    elliptic example. Each spectrum is evaluated at
     at the same random parameter $\ub$ drawn the prior. Numerical rank is the number
of singular values greater than some threshold $\epsilon \leq 1$. The misfit vector dimension ($N$
    in the original cost or $n$ in the RMA cost) is a hard upper bound on the
    numerical rank of the prior-preconditioned misfit Hessian. The
    misfit dimension $N$ for the original problem is 1333. }
\figlab{zoom}   
\end{figure}

Tables \tabref{tablesolves} and \tabref{tablesolves3D} respectively present algorithmic performance of
the original 2D and 3D elliptic problem compared to ten trials of the RMA with various distributions.
To investigate the effect of choosing the randomized misfit dimension $n
< r \leq N$ on work
complexity, $n$ is chosen to be $50$ for the 2D example where $N =
1333$, and $n=300$ for the 3D example where $N = 2474$. The Newton-CG
solver is terminated when the gradient, cost, or
step size falls below a tolerance of $10^{-6}$, or after 200 Newton
iterations for the 2D example and 15 Newton iterations for the 3D example. Each trial uses a different random number generator
seed. We observe that on average, using the RMA {\em with any
distribution} results in close to half as many PDE solves compared to solving the full
deterministic problem in the 2D example, and 14 to 28 percent fewer PDE
solves in the 3D example. There
appears to be little demonstrable difference in the quality of the
reconstruction as well; all experiments are successful in reconstructing the
Gaussian blob of high conductivity. Further
investigation on very large problems ($r=\mc{O}(1000)$ or larger) is needed. 

\begin{table}[h!t!b!]
\caption{Comparison of cost complexity measured in total number of
  PDE solves needed to resolve the 2D
  elliptic problem with a Gauss-Newton solver. AVG is average over ten
  trials. Convergence tolerances for the cost, gradient, and step size are set to
  $10^{-6}$ and the maximum number of Newton iterations allowed is 200.}
\tablab{tablesolves}
\resizebox{\columnwidth}{!}{%
\begin{tabular}{cccccccccccc}
\#PDE SOLVES & Trial 1 & Trial 2 & Trial 3 & Trial 4 & Trial 5 & Trial 6 & Trial 7 & Trial 8 & Trial 9 & Trial 10 & \textbf{AVG}\tabularnewline
\hline 
\hline 
Deterministic & 2423 & 2423 & 2423 & 2423 & 2423 & 2423 & 2423 & 2423 & 2423 & 2423 & \textbf{2423.0}\tabularnewline
Rademacher & 1303 & 1298 & 1273 & 1225 & 1279 & 1252 & 1270 & 1267 & 1252 & 1274 & \textbf{1269.3}\tabularnewline
Achlioptas & 1149 & 1253 & 1293 & 1266 & 1253 & 1245 & 1267 & 1262 & 1231 & 1254 & \textbf{1247.3}\tabularnewline
95-percent sparse & 1287 & 1272 & 1230 & 1273 & 1235 & 1217 & 1252 & 1248 & 1293 & 1238 & \textbf{1254.5}\tabularnewline
99-percent sparse & 1212 & 1243 & 1245 & 1247 & 1250 & 1263 & 1268 & 1226 & 1218 & 1274 & \textbf{1244.6}\tabularnewline
Gaussian & 1237 & 1258 & 1224 & 1226 & 1240 & 1273 & 1278 & 1255 & 1247 & 1234 & \textbf{1247.2}\tabularnewline
Uniform & 1217 & 1244 & 1233 & 1242 & 1241 & 1264 & 1259 & 1262 & 1275 & 1248 & \textbf{1248.5}\tabularnewline
\hline 
\end{tabular}
}
\end{table}

\begin{table}[h!t!b!]
\caption{Comparison of cost complexity measured in total number of
  PDE solves needed to resolve the 3D
  elliptic problem with a Gauss-Newton solver. AVG is average over ten
  trials. Convergence tolerances for the cost, gradient, and step size are set to
  $10^{-6}$ and the maximum number of Newton iterations allowed is 15.}
\tablab{tablesolves3D}
\resizebox{\columnwidth}{!}{%
\begin{tabular}{cccccccccccc}
\#PDE SOLVES & Trial 1 & Trial 2 & Trial 3 & Trial 4 & Trial 5 & Trial 6 & Trial 7 & Trial 8 & Trial 9 & Trial 10 & \textbf{AVG}\tabularnewline
\hline 
\hline 
Deterministic & 331 & 331 & 331 & 331& 331 & 331 & 331 &331 &331 & 331 & \textbf{331.0}\tabularnewline
Rademacher & 249&307&383&219&263&285&303&235&217&253&\textbf{271.4}\tabularnewline 
Achlioptas & 259&287&263&257&283&203&259&223&319&255&\textbf{260.8}\tabularnewline
95-percent sparse &279&273&221&231&285&241&255&313&217&367&\textbf{268.2}\tabularnewline
99-percent sparse& 277&321&271&249&279&395&273&223&281&279&\textbf{284.8}  \tabularnewline
Gaussian & 219&223&323&199&257&233&229&257&211&225&\textbf{237.6}\tabularnewline
Uniform &247&285&317&285&221&271&251&241&201&249& \textbf{256.8}\tabularnewline
\hline 
\end{tabular}
}
\end{table}

\section{Conclusions and future work}
\seclab{conclusions}

A randomized misfit approach is presented for reducing computational complexity induced by big data in
general large-scale PDE-constrained inverse
problems. The method permits a novel analysis of the stochastic cost
function and its minimizer via probabilistic bounds from random
projection theory. It is shown that a subgaussian distribution
guarantees the solution obtained from the randomized misfit approach
will satisfy Morozov's discrepancy principle with a low failure rate (that
decays exponentially with respect to the reduced dimension $n$).

It is shown that the stochastically derived method is equivalent to applying a random projection to the data misfit
vector. This results in a stochastic programming-based proof (up to
a constant) of a Johnson-Lindenstrauss lemma variant proved
previously (see, e.g. \cite{Alon03,
  JayramWoodruff13} for proofs based on combinatorics and
communication theory, respectively). Our connection provides two main theoretical insights.
The first is intuition into the surprising numerical accuracy with
small reduced misfit
dimension $n$. This phenomenon has been noted in related stochastic
methods, particularly in random source encoding methods, without
theoretical explanation. The second is an intuition into the
ubiquitous $\mc{O}(1/\sqrt{n})$ factor in Johnson-Lindenstrauss
transforms (a rate shown to be tight by
\cite{Alon03}) using a Monte Carlo framework.

The focus of this work is on the framework and resulting analysis of the
method. We presented results for a medium size ($N=\mc{O}(10^3)$) synthetic example
in 1D, 2D, and 3D and different distributions for numerical
justification of theoretical results and illustration of the
method. Results presented here are valid for nonlinear inverse problems with the exception
of part (ii) in Theorem \theoref{error} (which only applies to linear
forward models). We expect such a result is also true for nonlinear
inverse problems, and this is under investigation.

Combining dimension reduction and uncertainty quantification is the
broader focus of our ongoing work towards developing scalable methods for large-scale
inverse problems in high-dimensional parameter space with big
data. Our current research includes an application of the randomized
misfit approach to larger problems with big data, e.g. time-dependent data governed by expensive-to-solve forward models, and an
extension to the Bayesian solution. One project involves a large-scale
multi-tracer test inverse problem governed by an expensive-to-solve reservoir simulation. Also in
forthcoming tangential work we will compare different randomization frameworks
for solving inverse problems.


\section*{Acknowledgments}
We thank the anonymous referees for their valuable comments,
suggestions, and support. Their efforts helped improve the manuscript
significantly. We would like to thank Prof. Mark Girolami for pointing out the
similarity between randomized projections and the randomized
misfit approach, which led to the connection with Johnson-Lindenstrauss theory.
This in turn allowed us to carry out the analysis of the randomized
misfit approach presented here. We also thank Vishwas Rao for careful
proofreading. This research was partially supported by Department of
Energy (DOE) grants DE-SC0010518 and DE-SC0011118. We are grateful for the support.
\section*{References}

\bibliography{ceo} 

\end{document}

%% file: tbmacrosIP.tex
\def\etal{{\it et al.~}}
\newcommand{\A}{\mb{A}}
\newcommand{\alglab}[1]{\label{alg:#1}}
\newcommand{\algref}[1]{\ref{alg:#1}}
\newcommand{\att}{\left(t\right)}
\newcommand{\average}[1] {\ensuremath{\LRc{\!\{#1\}\!}}}
\newcommand{\B}{B}
\newcommand{\barr} {\begin{array}}
\newcommand{\Bb}{\mb{\B}}
\newcommand{\bb}{\mb{b}}
\newcommand{\bce}{\begin{center}}
\newcommand{\bde}{\begin{description}}
\newcommand{\bdm} {\begin{displaymath}}
\newcommand{\bea} {\begin{eqnarray}}
\newcommand{\bean} {\begin{eqnarray*}}
\newcommand{\ben} {\begin{enumerate}}
\newcommand{\beq} {\begin{equation}}
\newcommand{\bigO}{\mathcal{O}}
\newcommand{\bit}{\begin{itemize}}
\newcommand{\Cb}{\mb{C}} 
\newcommand{\C}{\mb{C}}
\newcommand{\cb}{\mb{c}}
\newcommand{\comp}[1]{\begin{pmatrix}{ #1 }\end{pmatrix}}
\newcommand{\corolab}[1]{\label{coro:#1}}
\newcommand{\cororef}[1]{\ref{coro:#1}}
\newcommand{\cp}{\ensuremath{{c_p}}}
\newcommand{\cs}{\ensuremath{{c_s}}}
\newcommand{\Cscrb}{\bs{\mscr{C}}}
\newcommand{\Cten}{\ensuremath{\tenfour{C}}}
\newcommand{\Curl} {\ensuremath{\nabla\times}}
\newcommand{\cvec}[1] {\ensuremath{\boldsymbol{{#1}}}}
\newcommand{\D}{\mb{D}}
\newcommand{\db}{\mb{d}}
\newcommand{\dbh}{\widehat{\db}}
\newcommand{\dd}[2] {\ensuremath{\frac{\partial {#1}}{\partial {#2}}}}
\newcommand{\DD}[2] {\ensuremath{\frac{d {#1}}{d {#2}}}}
\newcommand{\De} {\ensuremath{{\mathsf{D}^e}}}
\newcommand{\defilab}[1]{\label{defi:#1}}
\newcommand{\defiref}[1]{\ref{defi:#1}}
\newcommand{\Dehat} {\ensuremath{{\hat{\mathsf{D}}^e}}}
\newcommand{\Dei}[1]{\int_{\De} #1 \, d\vec{x}}
\newcommand{\DeI}[1]{\int_{D} #1 \, d\vec{x}}
\newcommand{\Deid}[1]{\int_{\De,N_e} #1 \, d\vec{x}}
\newcommand{\DeidN}[1]{\int_{\De,N} #1 \, d\vec{x}}
\newcommand{\DeiM}[1]{\int_{\Dhat} #1 \, d\vec{r}}
\newcommand{\DeiMd}[1]{\int_{\Dhat,N_e} #1 \, d\vec{r}}
\newcommand{\DeiMdN}[1]{\int_{\Dhat,N} #1 \, d\vec{r}}
\newcommand{\DeiMdNC}[2]{\int_{\Dhat,N_{#2}} #1 \, d\vec{r}}
\newcommand{\DeMd}[1]{\int_{\De,N_e} #1 \, d\vec{x}}
\newcommand{\DeMdN}[1]{\int_{D,N} #1 \, d\vec{x}}
\newcommand{\Dep} {\ensuremath{{\mathsf{D}^{e'}}}}
\newcommand{\Dhat} {\ensuremath{\hat{\mathsf{D}}}}
\newcommand{\diff}[1] {\ensuremath{\LRs{#1}}}
\newcommand{\Div} {\ensuremath{\nabla\cdot}}
\newcommand{\Divv}[1] {\ensuremath{\nabla_{#1}\cdot}}
\newcommand{\earr} {\end{array}}
\newcommand{\eb}{\mb{\e}}
\newcommand{\ece}{\end{center}}
\newcommand{\ede}{\end{description}}
\newcommand{\edm} {\end{displaymath}}
\newcommand{\edoc} {\end{document}}
\newcommand{\EE}{\ensuremath{\vec{E}}}
\newcommand{\eea} {\end{eqnarray}}
\newcommand{\eean} {\end{eqnarray*}}
\newcommand{\een} {\end{enumerate}}
\newcommand{\eeq} {\end{equation}}
\newcommand{\eit}{\end{itemize}}
\newcommand{\eqnlab}[1]{\label{eq:#1}}
\newcommand{\eqnref}[1]{\eqref{eq:#1}}
\newcommand{\etab}{\bs{\eta}}
\newcommand{\Ex}{{\mathbb{E}}}
\newcommand{\Expect}{\mathbb{E}}
\newcommand{\F}{{\mb{F}}}
\newcommand{\Fh}{{\widehat{\F}}}
\newcommand{\f}{f}
\newcommand{\fb}{\mb{\f}}
\newcommand{\FF}{\ensuremath{\boldsymbol{\mathfrak{F}}}}
\newcommand{\fmap}{\F}
\newcommand{\fmaph}{\widehat{\fmap}}
\newcommand{\fmapbar}{\tilde{\fmap}}
\newcommand{\figlab}[1]{\label{fig:#1}}
\newcommand{\figref}[1]{\ref{fig:#1}}
\newcommand{\G}{\mb{G}}
\newcommand{\Gb}{\mb{G}} 
\newcommand{\g}{g}
\newcommand{\gb}{\mb{\g}}
\newcommand{\gbf}[1] {\mathbf{#1}}
\newcommand{\GG}{\ensuremath{\ten{G}}}
\newcommand{\GM}[2]{\mc{N}\left( #1, #2 \right)}
\newcommand{\Gmc}{\mc{G}}
\newcommand{\Grad} {\ensuremath{\nabla}}
\newcommand{\Gradv}[1]{\ensuremath{\nabla_{#1}}}
\newcommand{\grbf}[1] {\mbox{\boldmath${#1}$\unboldmath}}
\newcommand{\Gscr}{\mscr{G}}
\newcommand{\Gscrb}{\bs{\Gscr}}
\newcommand{\h}{h}
\newcommand{\half} {\ensuremath{\frac{1}{2}}}
\newcommand{\halft} {\ensuremath{\frac{1}{2}}}
\newcommand{\hb}{\mb{\h}}
\newcommand{\Hc}{\mc{H}}
\newcommand{\HH}{\ensuremath{\ten{H}}}
\newcommand{\Hmc}{\mc{H}}
\newcommand{\Hscr}{\mscr{H}}
\newcommand{\Hscrb}{\bs{\Hscr}}
\newcommand{\Ht}{{\tilde{\H}}}
\newcommand{\I}{{\mb{I}}}
\newcommand{\iC}[1]{\int_{0}^{2\pi} #1 \, d\theta}
\newcommand{\iGb}[1]{\int_{\partial\Omega\setminus\Gamma_R} #1 \, ds}
\newcommand{\iGr}[1]{\int_{\Gamma_{\infty}} #1 \, ds}
\newcommand{\iGs}[1]{\int_{\Gamma_R} #1 \, ds}
\newcommand{\iGsy}[1]{\int_{\Gamma_{s}} #1 \, ds(\mb{y})}
\newcommand{\IN} {\ensuremath{{\mathcal{I}_N}}}
\newcommand{\In}{\hspace{0.2cm} \mbox{in} \hspace{0.2cm}}
\newcommand{\indicator}{\chi}
\newcommand{\INe} {\ensuremath{{\mathcal{I}_{N_e}}}}
\newcommand{\Int}[4]{\int_{#2}^{#3} #1 \, d{#4}}
\newcommand{\iOm}[1]{\int_{\Omega} #1 \, d\Omega}
\newcommand{\ipDed}[3]{\LRp{ #1, #2}_{\De,N}^{#3}}
\newcommand{\J}{J}
\newcommand{\Jn}{\J_n}
\newcommand{\jump}[1] {\ensuremath{\LRs{\![#1]\!}}}
\newcommand{\K}{\mb{K}}
\newcommand{\lamh}{\hat{\lambda}}
\newcommand{\lamt}{\check{\lambda}}
\newcommand{\lamth}{\lambda^3}
\newcommand{\lamtt}{\tilde{\lamt}}
\newcommand{\lamtth}{\lambda^{2,3}}
\newcommand{\lemlab}[1]{\label{lem:#1}}
\newcommand{\lemref}[1]{\ref{lem:#1}}
\newcommand{\like}{{\pi_{\text{like}}\LRp{\db|\u}}}
\newcommand{\likeb}{{\pi_{\text{like}}\LRp{\db|\ub}}}
\newcommand{\LRa}[1]{\left< #1 \right>}
\newcommand{\LRc}[1]{\left\{ #1 \right\}}
\newcommand{\LRp}[1]{\left( #1 \right)}
\newcommand{\LRs}[1]{\left[ #1 \right]}
\newcommand{\Ltwo}{{L^2\LRp{\Omega}}}
\newcommand{\Ltwoc}{{{L}^2\LRp{\Omega}}}
\newcommand{\M}{\mb{M}}
\newcommand{\Maximize}{\hspace{0.2cm} \mbox{maximize} \hspace{0.2cm}}
\newcommand{\mb}[1]{\mathbf{#1}}
\newcommand{\mbb}[1]{\mathbb{#1}}
\newcommand{\mc}[1]{\mathcal{#1}}
\newcommand{\Mc}{{\M}}
\newcommand{\mcC}[1]{\mathcal{#1}}
\newcommand{\mf}[1]{\mathfrak{#1}}
\newcommand{\mi}[1]{\mathit{#1}}
\newcommand{\Minimize}{\hspace{0.2cm} \mbox{minimize} \hspace{0.2cm}}
\newcommand{\misfit}{\fmap\LRp{\ub}-\db}
\newcommand{\misfitonly}{\yobsh-\fmaph(\ub)}
\newcommand{\misfitterm}{\frac{1}{2\sigma^2} |\db-\fmap(\ub)|^2}
\newcommand{\mscr}[1]{\mathscr{#1}}
\newcommand{\Ned} {\ensuremath{{N_\text{ed}}}}
\newcommand{\Nel} {\ensuremath{{N_\text{el}}}}
\newcommand{\nn}{\ensuremath{\vec{n}}}
\newcommand{\nor}[1]{\left\| #1 \right\|}
\newcommand{\nxnx}[1] {\ensuremath{\nn\times\left(\nn\times{#1}\right)}}
\newcommand{\Oi}[1]{\int_{\mc{D}} #1 \, d\vec{x}}
\newcommand{\On}{\hspace{0.2cm} \mbox{on} \hspace{0.2cm}}
\newcommand{\params}{\ub}
\newcommand{\pb}{\mb{p}}
\newcommand{\Pb}{\mb{P}}
\newcommand{\pDei}[1]{\int_{\partial \De} #1 \, d\vec{x}}
\newcommand{\pDeid}[1]{\int_{\partial \De,N_e} #1 \, d\vec{x}}
\newcommand{\pDeidN}[1]{\int_{\partial \De,N} #1 \, d\vec{x}}
\newcommand{\pDeiM}[1]{\int_{\partial \Dhat} #1 \, d\vec{r}}
\newcommand{\pDeiMd}[1]{\int_{\partial \Dhat,N_e} #1 \, d\vec{r}}
\newcommand{\pDeiMdN}[1]{\int_{\partial \Dhat,N} #1 \, d\vec{r}}
\newcommand{\pDeiMdNC}[2]{\int_{\partial \Dhat,N_{#2}} #1 \, d\vec{r}}
\newcommand{\pMortar}[2]{\int_{\mc{M}_{#2}} #1 \, d\vec{r}}
\newcommand{\pMortard}[2]{\int_{\mc{M}^{#2},N_{#2}} #1 \, d\vec{r}}
\newcommand{\pMortardd}[3]{\int_{\mc{M}^{#2},N_{#3}} #1 \, d\vec{r}}
\newcommand{\pMortardh}[3]{\int_{\mc{M}_{#2},N_{#3}} #1 \, d\vec{r}}
\newcommand{\pMortardhn}[4]{\int_{{#2}_{#3},N_{#4}} #1 \, d\vec{r}}
\newcommand{\point}[1] {\ensuremath{\boldsymbol{#1}}}
\newcommand{\pOmega}{\partial\Omega}
\newcommand{\post}{{\pi\LRp{\ub | \db}}}
\newcommand{\postu}{{\pi\LRp{\u | \db}}}
\newcommand{\postv}{{\pi\LRp{\vb | \db}}}
\newcommand{\pp}[2]{\frac{\partial #1}{\partial #2}}
\newcommand{\ppcp}[1]{\frac{\partial #1}{\partial \vec{c_p}}}
\newcommand{\ppcpj}[1]{\frac{\partial #1}{\partial c_p}}
\newcommand{\ppcpm}[1]{\frac{\partial #1}{\partial \vec{c_p}^-}}
\newcommand{\ppcpmf}[1]{\frac{\partial #1}{\partial \vec{\mf{c_p}}}}
\newcommand{\ppcpp}[1]{\frac{\partial #1}{\partial \vec{c_p}^+}}
\newcommand{\ppcs}[1]{\frac{\partial #1}{\partial \vec{c_s}}}
\newcommand{\ppcsm}[1]{\frac{\partial #1}{\partial \vec{c_s}^-}}
\newcommand{\ppcsp}[1]{\frac{\partial #1}{\partial \vec{c_s}^+}}
\newcommand{\ppho}[3]{\frac{\partial^{#3} #1}{{\partial #2}^{#3}}}
\newcommand{\ppm}[1]{\frac{\partial #1}{\partial \vec{m}}}
\newcommand{\ppmi}[1]{\frac{\partial #1}{\partial m_i}}
\newcommand{\ppmj}[1]{\frac{\partial #1}{\partial m_j}}
\newcommand{\prior}{{\pi_{\text{prior}}\LRp{\ub}}}
\newcommand{\prioru}{{\pi_{\text{prior}}\LRp{\u}}}
\newcommand{\Prob}{\mathbb{P}}
\newcommand{\propolab}[1]{\label{propo:#1}}
\newcommand{\proporef}[1]{\ref{propo:#1}}
\newcommand{\Q}{\mc{Q}}
\newcommand{\q}{q}
\newcommand{\R}{\mathbb{R}}
\newcommand{\rb}{\mb{\r}}
\newcommand{\RE}{\mbox{{I}}\hspace{-0.5ex}\mbox{{R}}}
\newcommand{\remalab}[1]{\label{rema:#1}}
\newcommand{\remaref}[1]{\ref{rema:#1}}
\newcommand{\Rvert}[1]{\left. #1 \right|}
\newcommand{\Sb}{\mb{\S}}
\newcommand{\seclab}[1]{\label{sect:#1}}
\newcommand{\secref}[1]{\ref{sect:#1}}
\newcommand{\Sigb}{\bs{\Sigma}}
\newcommand{\snor}[1]{\left| #1 \right|}
\newcommand{\Sp}{\ensuremath{{S^e_p}}}
\newcommand{\Ss}{\ensuremath{{S^e_s}}}
\newcommand{\SubjectTo}{\hspace{0.2cm} \mbox{subject} \hspace{0.15cm}\mbox{to} \hspace{0.2cm}}
\newcommand{\T}{T}
\newcommand{\tablab}[1]{\label{tab:#1}}
\newcommand{\tabref}[1]{\ref{tab:#1}}
\newcommand{\target}{\pi}
\newcommand{\tcs}{\ensuremath{{\tilde{c}_s}}}
\newcommand{\td}[2]{\frac{d #1}{d #2}}
\newcommand{\TDei}[1]{\int_{0}^T\int_{\De} #1 \, d\vec{x}dt}
\newcommand{\TDeid}[1]{\int_{0}^T\int_{\De,N} #1 \, d\vec{x}dt}
\newcommand{\tEE}{\ensuremath{\tilde{\vec{E}}}}
\newcommand{\ten}[1] {\ensuremath{\boldsymbol{#1}}}
\newcommand{\tenfour}[1] {\ensuremath{\boldsymbol{\mathsf{#1}}}}
\newcommand{\theolab}[1]{\label{theo:#1}}
\newcommand{\theoref}[1]{\ref{theo:#1}}
\newcommand{\Ti}[1]{\int_{0}^T #1 \, dt}
\newcommand{\TODO}[1]{ \fbox{\parbox{3in}{\bf TODO: #1}}}
\newcommand{\TOi}[1]{\int_{0}^T \int_{\mc{D}} #1 \, d\vec{x}dt}
\newcommand{\TpDei}[1]{\int_{0}^T\int_{\partial \De} #1 \, d\vec{x}}
\newcommand{\TpDeid}[1]{\int_{0}^T\int_{\partial \De,N} #1 \, d\vec{x}}
\newcommand{\tSS}{\ensuremath{\tilde{\ten{S}}}}
\newcommand{\tvv}{\ensuremath{\tilde{\vec{v}}}}
\newcommand{\ub}{\mb{\u}}
\newcommand{\ubo}{\tilde{\ub}}
\newcommand{\ubt}{\check{\ub}}
\newcommand{\umap}{\ub^\star}
\newcommand{\umapn}{\umap_n}
\newcommand{\uo}{\tilde{\u}}
\newcommand{\ut}{\check{\u}}
\newcommand{\uth}{\u^3}
\newcommand{\V}{V}
\newcommand{\varepsb}{\bs{r}}
\newcommand{\vb}{\mb{\v}}
\newcommand{\Vb}{\mb{\V}}
\newcommand{\vectornorm}[1]{\left|\left|#1\right|\right|}
\newcommand{\Vp}{\ensuremath{{V^e_p}}}
\newcommand{\vr}{\left(\ensuremath{\vec{r}}\right)}
\newcommand{\Vs}{\ensuremath{{V^e_s}}}
\newcommand{\vv}{\ensuremath{\vec{v}}}
\newcommand{\vvec}[1] {\ensuremath{\boldsymbol{#1}}}
\newcommand{\w}{w}
\newcommand{\W}{W}
\newcommand{\wb}{\mb{\w}}
\newcommand{\wbd}{\mb{\w}}
\newcommand{\Wb}{\mb{\W}}
\newcommand{\wh}{\hat{\w}}
\newcommand{\wt}{\check{\w}}
\newcommand{\wth}{\w^3}
\newcommand{\wtth}{\w^{2,3}}
\newcommand{\ww}{\ensuremath{\vec{w}}}
\newcommand{\x}{\mb{x}}
\newcommand{\X}{X}
\newcommand{\yb}{\mb{y}}
\newcommand{\yobs}{\db}
\newcommand{\yobsh}{\widehat{\yobs}}
\newcommand{\yobsbar}{\tilde{\yobs}}
\newcommand{\z}{z}
\newcommand{\zb}{\mb{\z}}
\newcommand*{\numbercell}{\stepcounter{tablecell}\thetablecell.}
\newcounter{tablecell}
\newtheorem{corollary}{Corollary}
\newtheorem{definition}{Definition}
\newtheorem{lemma}{Lemma}
\newtheorem{proposition}{Proposition}
\newtheorem{remark}{Remark}
\newtheorem{theorem}{Theorem}
\providecommand{\abs}[1]{\left\lvert#1\right\rvert}
\providecommand{\norm}[1]{\left\lVert#1\right\rVert}
\renewcommand{\b}{b}
\renewcommand{\bs}[1]{\boldsymbol{#1}}
\renewcommand{\d}{d}
\renewcommand{\e}{e}
\renewcommand{\H}{\mb{H}}
\renewcommand{\L}{{\mb{L}}}
\renewcommand{\P}{P}
\renewcommand{\r}{r}
\renewcommand{\S}{\mb{S}}
\renewcommand{\S}{S}
\renewcommand{\sb}{\mb{s}}
\renewcommand{\SS}{\ensuremath{\ten{S}}}
\renewcommand{\u}{u}
\renewcommand{\v}{v}
\renewcommand{\vec}[1] {\ensuremath{\boldsymbol{#1}}}
\newcommand{\eval}[2][\right]{\relax  \ifx#1\right\relax \left.\fi#2#1\rvert}